\documentclass[11pt,leqno]{amsart}

\usepackage{amsmath}
\usepackage{amssymb}
\usepackage{a4wide}
\usepackage{graphics}
\usepackage{epsfig}
\usepackage{enumerate}
\newtheorem{theorem}{Theorem}
\newtheorem{lemma}[theorem]{Lemma}
\newtheorem{proposition}[theorem]{Proposition}
\newtheorem{definition}[theorem]{Definition}
\newtheorem{remark}[theorem]{Remark}

\newcommand{\Subsection}[1]{\subsection{ #1} ${}^{}$}

\newcounter{hypo}

\def\C{{\mathbb C}}

\def\N{{\mathbb N}} 
\def\R{{\mathbb R}}

\def\CO{\mathcal {O}}

\def\one{{\mathchoice {\rm 1\mskip-4mu l} {\rm 1\mskip-4mu l} {\rm 1\mskip-4.5mu l} {\rm 1\mskip-5mu l}}}

 \def\im{\mathop{\rm Im}\nolimits}

\def\oph{\mathop{\rm Op}_{h}\nolimits}

\def\dive{\mathop{\rm div}\nolimits}

\def\supp{\mathop {\rm supp}\nolimits}

\def\ad{\mathop{\rm ad}\nolimits}
\def\<{\langle}
\def\>{\rangle}

\newcommand{\fract}[2]{\genfrac{}{}{0pt}{}{\scriptstyle #1}{\scriptstyle #2}}

\author[J.-F. Bony]{Jean-Fran\c{c}ois Bony}
\address{Jean-Fran\c{c}ois Bony, Institut de Math\'ematiques de Bordeaux, UMR 5251 du CNRS, Universit\'e de Bordeaux I, 351 cours de la Lib\'eration, 33405 Talence cedex, France}
\email{bony@math.u-bordeaux1.fr}

\author[D. H\"{a}fner]{Dietrich H\"{a}fner}
\address{Dietrich H\"{a}fner, Universit\'e de Grenoble 1, Institut Fourier, UMR 5582 du CNRS, BP 74, 38402 St Martin d'H\`eres, France}
\email{Dietrich.Hafner@ujf-grenoble.fr}

\title[Local energy decay]{Local energy decay for several evolution equations on asymptotically euclidean manifolds}



\begin{document}

\begin{abstract}
Let $P$ be a long range metric perturbation of the Euclidean Laplacian on $\R^d$, $d \geq 2$. We prove local energy decay for the solutions of the wave, Klein--Gordon and Schr\"odinger equations associated to $P$. The problem is decomposed in a low and high frequency analysis. For the high energy part, we assume a non trapping condition. For low (resp. high) frequencies we obtain a general result about the local energy decay for the group $e^{i t f(P)}$ where $f$ has a suitable development at zero (resp. infinity). 
\end{abstract}

\maketitle

\section{Introduction}

This paper is devoted to the study of the local energy decay for several evolution equations associated to long range metric perturbations of the Euclidean Laplacian on $\R^d$, $d \geq 2$. In particular, we show that the local energy for the wave (resp. Schr\"{o}dinger) equation decays like $\< t \>^{1-d + \varepsilon}$ (resp. $\vert t \vert^{- d/2} \< t \>^{\varepsilon}$). The restriction on the decays comes from the low frequency part, whose study constitutes the main part of the paper.

In the case of compactly supported (or exponentially decaying) perturbations, one of the most efficient approaches to prove local energy decay is the theory of resonances. For the wave equation outside some non trapping obstacles in odd dimension $\geq 3$, Lax and Phillips \cite{LaPh89_01} have obtained an exponential decay of the local energy. This has been generalized by Lax, Morawetz and Phillips \cite{LaMoPh63_01} to star-shaped obstacles and by Melrose and Sj\"{o}strand \cite{MeSj78_01} to non trapping obstacles. For general, non trapping, differential operators, Va{\u\i}nberg \cite{Va89_01} has obtained exponential (resp. polynomial) decay in odd (resp. even) dimensions. His proof rests on estimates of the cut-off resolvent in the complex plane. The theory of resonances can also be used to analyze the trapping situation; but, in this case, Ralston \cite{Ra69_01} has proved that there is necessarily a loss of derivatives in the estimates. We mention the work of Burq \cite{Bu98_01} without assumption on the trapped set, the work of Tang and Zworski \cite{TaZw00_01} for the resonances close the real line, the work of Christianson \cite{Ch09_01} for hyperbolic trapped sets with small topological pressure studied by Nonnenmacher and Zworski \cite{NoZw09_01} and the work of Petkov and Stoyanov \cite{PeSt10_01} outside several disjoint convex compact obstacles.

For slowly decaying perturbations, it is not clear how to use the theory of resonances. Instead, one can apply other methods (like resolvent estimates, perturbation theory, vector field methods, \ldots) giving typically polynomial decays. Jensen, Mourre and Perry \cite{JeMoPe84_01} and Hunziker, Sigal and Soffer \cite{HuSiSo99_01} have proved abstract local energy decays using Mourre theory. There is also a huge literature concerning the local energy decay for the Schr\"{o}dinger equation perturbed by a potential. We only mention here the works of Rauch \cite{Ra78_01} and of Jensen and Kato \cite{JeKa79_01}. Perturbation theory can also be applied for short range metric perturbations as in the work of Wang \cite{Wa06_01}. Schlag, Soffer and Staubach \cite{ScSoSt10_01,ScSoSt10_02} (see also the references therein) have considered radial short range perturbations of conical ends.

There has been important progress concerning the local energy decay for the wave equation in black hole type space-times. Finster, Kamran, Smoller and Yau \cite{FiKaSmYa06_01}, Tataru and Tohaneanu \cite{TaTo08_01}, Dafermos and Rodnianski \cite{DaRo08_01}, Andersson and Blue \cite{AnBl09_01} and Tataru \cite{Ta09_01} have proved various results in this direction for the Kerr metric which is, far away from the black hole, a long range perturbation of the Minkowski metric.

In dimension $3$, Tataru \cite{Ta09_01} has obtained a $\< t \>^{-3}$ local decay rate for some wave equations with long range perturbations which are radial up to short range terms. In our long range setting, Bouclet \cite{Bo10_01} has established estimates for various evolution equations with other decay rates. Note that he also obtained low frequency estimates for powers of the resolvent (see also Bouclet \cite{Bo08_01} and our work \cite{BoHa10_02} for estimates on the resolvent at low energy and Guillarmou and Hassell \cite{GuHa08_01,GuHa09_01} for a low frequency description of the resolvent using pseudodifferential calculus).

One can also consider evolution equations of higher order. For example, Ben-Artzi, Koch and Saut \cite{BeKoSa00_01} have established different dispersive estimates for the fourth order Schr\"{o}dinger groups $e^{i t ( \varepsilon \Delta + \Delta^{2} )}$ with $\varepsilon \in \{ -1 , 0 ,1 \}$. Moreover, Balabane \cite{Ba89_01} has obtained smoothing effects and local energy decays for evolution equations associated to an elliptic Fourier multiplier.

We consider the following operator on $\R^d$, with $d \geq 2$,
\begin{equation} \label{a5}
P= - b \dive ( G \nabla b ) = - \sum_{i,j=1}^{d} b(x) \frac{\partial \ }{\partial x_{i}} G_{i,j} (x) \frac{\partial \ }{\partial x_{j}} b (x) ,
\end{equation}
where $b(x)\in C^{\infty}(\R^d)$ and $G(x)\in C^{\infty}(\R^d;\R^{d\times d})$ is a real symmetric $d\times d$ matrix. The $C^{\infty}$ hypothesis is made mostly for convenience, much weaker regularity could actually be considered. We make an ellipticity assumption:
\begin{equation} \tag{H1} \label{a3}
\exists C>0, \ \forall x\in \R^d \qquad G(x)\geq C I_d \ \text{ and } \  b(x)\geq C, 
\end{equation}
$I_d$ being the identity matrix on $\R^{d}$. We also assume that $P$ is a long range perturbation of the Euclidean Laplacian:
\begin{equation} \tag{H2} \label{a4}
\exists \rho > 0 , \ \forall \alpha \in \N^d \qquad \vert \partial^{\alpha}_x(G(x)-I_d) \vert + \vert \partial^{\alpha}_x ( b(x) - 1 ) \vert\lesssim \<x\>^{-\rho-\vert\alpha\vert}. \\
\end{equation}

In particular, if $b=1$, we are concerned with an elliptic operator in divergence form $P = - \dive (G \nabla )$. On the other hand, if $G=(g^2 g^{i,j}(x))_{i,j},\, b=(\det g^{i,j})^{1/4},\, g=\frac{1}{b}$, then the above operator is unitarily equivalent to the Laplace--Beltrami $- \Delta_{\mathfrak{g}}$ on $(\R^d, \mathfrak{g})$ with metric
\[\mathfrak{g} = \sum_{i,j=1}^{d} g_{i,j} (x) \, d x^i \, d x^j ,\]
where $(g_{i,j})_{i,j}$ is inverse to $(g^{i,j})_{i,j}$ and the unitary transform is just multiplication by $g$. In the following, $\Vert \cdot \Vert$ will always design the norm on $L^{2} ( \R^{d} )$. We first obtain local energy decay estimates for several evolution equations at low frequency.

\pagebreak

\begin{theorem}\sl \label{th1}
Assume \eqref{a3}--\eqref{a4} and $d \geq 2$. For all $\chi \in C_0^{\infty}( \R )$ and $\varepsilon > 0$, we have

$i)$ for the wave equation
\begin{gather} \label{wave1}
\Big\Vert\<x\>^{1-d}\frac{\sin t\sqrt{P}}{\sqrt{P}}\chi(P) \<x\>^{1-d}\Big\Vert \lesssim \<t\>^{1-d+ \varepsilon},\\
\Big\Vert\<x\>^{-d}(\partial_t,\sqrt{P})\frac{\sin t\sqrt{P}}{\sqrt{P}}\chi(P) \<x\>^{-d}\Big\Vert \lesssim \<t\>^{-d+\varepsilon}. \label{wave2}
\end{gather}

$ii)$ for the Klein--Gordon equation
\begin{equation} \label{KG}
\big\Vert \<x\>^{-d/2}e^{it\sqrt{1+P}} \chi(P) \<x\>^{-d/2}\big\Vert \lesssim \<t\>^{-d/2+\varepsilon}.
\end{equation}

$iii)$ for the Schr\"odinger equation
\begin{equation} \label{Schr}
\big\Vert \<x\>^{-d/2}e^{itP} \chi(P) \<x\>^{-d/2}\big\Vert \lesssim \<t\>^{-d/2+\varepsilon}.
\end{equation}

$iv)$ for the fourth order Schr\"odinger equation
\begin{gather} \label{4Schr1}
\big\Vert \<x\>^{-d/2}e^{it(P+P^2)} \chi(P) \<x\>^{-d/2}\big\Vert \lesssim \<t\>^{-d/2+\varepsilon},\\
\label{4Schr2}
\big\Vert \<x\>^{-d/2}e^{itP^2} \chi(P) \<x\>^{-d/2}\big\Vert \lesssim \<t\>^{-d/4+\varepsilon}.
\end{gather}
\end{theorem}

The above theorem collects special cases of a more general theorem.
\begin{theorem}\sl \label{th4}
Assume \eqref{a3}--\eqref{a4} and $d \geq 2$. Let $f$ be a real function such that
\begin{equation*}
f (x) = a_{0} + a_{1} x^{\alpha} + x^{\alpha + \nu} g (x) ,
\end{equation*}
with $a_{1} \neq 0$, $\alpha , \nu > 0$ and $g \in C^{\infty} ( \R )$. Let $\chi \in C^{\infty}_{0} ( \R )$ be such that $f^{\prime} (x) > 0$ for all $x \in \supp \chi \cap ] 0 , + \infty [$.

$i)$ If $0 < \alpha \leq 1$, we have for all $\varepsilon >0$
\begin{equation} \label{b5i}
\big\Vert \< x \>^{-\frac{d}{2 \alpha}} e^{i t f (P)} \chi(P) \< x \>^{-\frac{d}{2 \alpha}} \big\Vert \lesssim \< t \>^{- \frac{d}{2 \alpha} + \varepsilon} .
\end{equation}

$ii)$ If $\alpha > 1$, we have for all $\varepsilon >0$
\begin{equation} \label{b5ii}
\big\Vert \< x \>^{- \frac{d}{2}} e^{i t f (P)} \chi(P) \< x \>^{- \frac{d}{2}} \big\Vert \lesssim \< t \>^{- \frac{d}{2 \alpha} + \varepsilon}.
\end{equation}
\end{theorem}

We now give estimates which are global in energy. Since the Hamiltonian trajectories play a crucial role at high frequencies, the local energy decay depends on the geometry of these curves. In this paper, we will assume that
\begin{equation} \tag{H3} \label{a51}
P \text{ is non-trapping.}
\end{equation}
Under this assumption, we obtain local energy decay estimates for various evolution equations.

\begin{theorem}\sl \label{glest}
Assume \eqref{a3}--\eqref{a51} and $d \geq 2$. For all $\varepsilon > 0$, we have

$i)$ for the wave equation
\begin{gather} \label{wave3}
\Big\Vert \< x \>^{1-d} \frac{\sin t\sqrt{P}}{\sqrt{P}} u \Big\Vert_{H^1 ( \R^{d} )} \lesssim \< t \>^{1-d+ \varepsilon} \big\Vert \< x \>^{d-1} u \Vert ,    \\
\Big\Vert \< x \>^{-d} (\partial_t , \sqrt{P} ) \frac{\sin t \sqrt{P}}{\sqrt{P}} u \Big\Vert \lesssim \< t \>^{-d+ \varepsilon} \big\Vert \< x \>^d u \big\Vert . \label{wave4}
\end{gather}

$ii)$ for the Klein--Gordon equation
\begin{equation} \label{KG1}
\big\Vert \<x\>^{-d/2} e^{i t \sqrt{1+P}} u \big\Vert \lesssim \< t \>^{-d/2 + \varepsilon} \big\Vert \<x\>^{d/2} u \big\Vert .
\end{equation}

$iii)$ for the Schr\"odinger equation
\begin{equation} \label{Schr1}
\big\Vert \<x\>^{-d/2} e^{i t P} u \big\Vert \lesssim \vert t \vert^{-d/2} \< t \>^{\varepsilon} \big\Vert \<x\>^{d/2} u \big\Vert_{H^{-d/2} ( \R^{d} )} .
\end{equation}

$iv)$ for the fourth order Schr\"odinger equation
\begin{gather} \label{4Schr3}
\big\Vert \< x \>^{-d/2} e^{i t (P+P^2)} u \big\Vert \lesssim \vert t \vert^{-d/2} \< t \>^{\varepsilon} \big\Vert \<x\>^{d/2} u \big\Vert_{H^{- 3 d / 2} ( \R^{d} )} ,  \\
\label{4Schr4}
\big\Vert \< x \>^{-d/2} e^{i t P^2} u \big\Vert \lesssim \vert t \vert^{-d/2} \< t \>^{d/4 + \varepsilon} \big\Vert \< x \>^{d/2} u \big\Vert_{H^{- 3 d /2} ( \R^{d} )} .
\end{gather}
\end{theorem}

\begin{remark}\rm
$i)$ In even dimensions, \eqref{wave3} and \eqref{wave4} are optimal modulo the loss of $\< t \>^{\varepsilon}$. Indeed, the fundamental solution of the wave equation on the Minkowski space is explicitly known and a better estimate is not possible (see e.g. Section 3.5 of \cite{Ta96_01}).

$ii)$ The type of decay we obtain for the wave equation does not depend on the parity of the dimension. This is not the case on the Minkowski space since the strong Huygens principle assures that the local energy decays as much as we want for $d \geq 3$ odd. For compactly supported perturbation, the theory of resonances gives an exponential decay (see e.g. Vainberg \cite{Va89_01}). The difference with our results is that, roughly speaking, we only use upper bounds on the kernel of the resolvent (which do not depend on the parity of the dimension) and not analytic properties (only valid in odd dimensions).

$iii)$ The decays that we obtain globally are limited by the best possible decays for the low frequency part given by Theorem \ref{th1}. But, outside of the low frequencies, better decays follow from the estimates \eqref{a26} and \eqref{a27} below which hold in all dimension $d \geq 1$.

$iv)$ It is perhaps also possible to deal with some trapping situations for which the high frequency behavior of the evolution is well understood. Note that, in this case, there will necessarily be a loss of derivatives. For example, one can hope to remove the assumption that the perturbation is compactly supported in the work of Christianson \cite{Ch09_01} to obtain polynomial local energy decay (limited by the decay at low frequency) for hyperbolic trapped sets.
\end{remark}

Theorem \ref{glest} comes from Theorem \ref{th1} and the following general result at high frequency.
\begin{theorem}\sl \label{a28}
Assume \eqref{a3}--\eqref{a51} and $d \geq 1$. Let $f$ be a real function such that, for $x \geq 1$,
\begin{equation*}
f (x ) = x^{\alpha} + x^{\alpha - \nu} g (x),
\end{equation*}
with $\alpha , \nu >0$ and $g \big( \frac{1}{x} \big) \in C^{\infty} ( [0, 1 [ )$.

$i)$ For all $\varphi \in C^{\infty}_{0} ( ] 0 , + \infty [)$ and $\mu \geq 0$, we have
\begin{equation} \label{a26}
\big\Vert \<x\>^{- \mu} e^{i t f (P)} \varphi ( h^{2} P ) \<x\>^{- \mu} \big\Vert \lesssim \big\< t h^{1 - 2 \alpha} \big\>^{- \mu} ,
\end{equation}
uniformly for $h>0$ small enough and $t \in \R$.

$ii)$ For all $\chi \in C^{\infty}_{0} ( \R )$ equal to $1$ on a sufficiently large neighborhood of $0$ and $\mu \geq 0$,
\begin{equation} \label{a27}
\big\Vert \<x\>^{- \mu} e^{i t f(P)} ( 1 - \chi ) (P) u \big\Vert_{L^{2} ( \R^{d} )} \lesssim
\left\{ \begin{aligned}
&\< t \>^{- \mu} \big\Vert \< x \>^{\mu} u \big\Vert_{H^{\mu - 2 \alpha \mu} ( \R^{d} )} &&\text{for } \alpha \leq 1/2 , \\
&\vert t \vert^{- \mu} \big\Vert \< x \>^{\mu} u \big\Vert_{H^{\mu - 2 \alpha \mu} ( \R^{d} )} &&\text{for } \alpha > 1/2 .
\end{aligned} \right.
\end{equation}
\end{theorem}

\begin{remark}\rm
The proof of Theorem \ref{a28} rests on a semiclassical argument which can be used in other situations. For example, Proposition \ref{a15} gives another proof of one of the results obtained by Wang in \cite{Wa87_01}: Let $P_{h} = - h^{2} \Delta + V(x)$ be a semiclassical Schr\"{o}dinger operator with a potential $V$ satisfying $\vert \partial_{x}^{\alpha} V (x) \vert \lesssim \< x \>^{- \rho - \vert \alpha \vert}$ for some $\rho >0$ and all $\alpha \in \N^{d}$. Assume that $[ a , b ] \subset ] 0 , + \infty [$ is an interval of non trapping energy for $p = \xi^{2} + V (x)$. Then, for all $\varphi \in C^{\infty}_{0} ( [ a ,b ])$ and $\mu \geq 0$, we have
\begin{equation*}
\big\Vert \< x \>^{- \mu} e^{i t P_{h} /h} \varphi ( P_{h} ) \< x \>^{- \mu} \big\Vert \lesssim \< t \>^{- \mu} .
\end{equation*}
uniformly for $t \in \R$.
\end{remark}

To prove the different local energy decays, we use an abstract result of Hunziker, Sigal and Soffer \cite{HuSiSo99_01} based on a Mourre estimate. Then we have to make a specific study of the commutator estimates at low and high frequencies. The low energies are treated in Section \ref{a52} using Hardy type estimates. The high frequency results follow from a local energy decay for general semiclassical operators, using only commutators estimates, proved in Section \ref{a53}. Eventually, Section \ref{a1} collects low frequency resolvent estimates generalizing those of \cite[Appendix B]{BoHa10_01}.

\section{Low frequency estimates} \label{a52}

\Subsection{Abstract setting}

In this section we recall a theorem of Hunziker, Sigal and Soffer \cite{HuSiSo99_01} that will be used in the following. To do so, we have to recall the notion of regularity with respect to an operator. A full presentation of this theory can be found in the book of Amrein, Boutet de Monvel and Georgescu \cite{AmBoGe96_01}.

\begin{definition}\sl
Let $(A, D(A))$ and $(H,D(H))$ be self-adjoint operators on a separable Hilbert space ${\mathcal H}$. The operator $H$ is of class $C^{k} (A)$ for $k >0$, if there is $z \in \C \setminus \sigma ( H)$ such that
\begin{equation*}
\R \ni t \longrightarrow e^{i t A} (H-z)^{-1} e^{- i t A},
\end{equation*}
is $C^{k}$ for the strong topology of ${\mathcal L} ( {\mathcal H} )$.
\end{definition}

Let $H \in C^{1} (A)$ and $I \subset \sigma (H)$ be an open interval. We assume that $A$ and $H$ satisfy a Mourre estimate on $I$:
\begin{equation} \label{ME}
\one_{I} (H) i [ H , A ] \one_{I} (H) \geq \delta \one_{I} (H) ,
\end{equation}
for some $\delta >0$. As usual, we define the multi-commutators $\ad_{A}^{j} B$ inductively by $\ad_{A}^{0} B =B$ and $\ad_{A}^{j+1} B = [ A , \ad_{A}^{j} B ]$.

In the sequel, we will need a version of a result of \cite{HuSiSo99_01} which holds uniformly in the operators $A$ and $H$. Following the different constants in that paper, we obtain the result below supposing that the Mourre estimate \eqref{ME} is satisfied uniformly (i.e. the constant $\delta >0$ and the interval $I$ are fixed).

\begin{theorem}[Hunziker, Sigal, Soffer]\sl \label{a16}
Let $\mu >0$ and $\overline{\mu} = \min \{ n \in \N ; \, n > \mu +1 \}$. Let $H,A$ be two self-adjoint operators such that $H \in C^{\overline{\mu}} (A)$, that the Mourre estimate \eqref{ME} holds and that the commutators $\ad^{j}_{A} H$ are bounded for $1 \leq j \leq \overline{\mu}$. Then, for all $\chi \in C_0^{\infty} ( I )$,
\begin{equation*}
\big\Vert \< A \>^{- \mu} e^{i t H} \chi (H) u \big\Vert \leq  P_{\mu , \chi} \big( \Vert \ad^1_A H \Vert, \ldots , \Vert \ad^{\overline{\mu}}_A H \big\Vert \big) \< t \>^{- \mu} \big\Vert \<A\>^{\mu} u \big\Vert , 
\end{equation*}  
where $P_{\mu , \chi}$ is a polynomial.
\end{theorem}

\Subsection{The Mourre estimate} \label{SME}

The goal of this section is to obtain some estimates at low frequency that will be used in the next section to prove Theorems \ref{th1} and \ref{th4} thanks to Theorem \ref{a16}. Note that Vasy and Wunsch \cite{VaWu09_01} have shown Mourre estimates for $P$ and $\sqrt{P}$ at low frequency for scattering manifolds. For that, they have used some Hardy type estimates proved by pseudodifferential calculus.

Let $\psi_{\lambda} \in C^{\infty} ( \R ; \R )$ be a family of functions fulfilling the following conditions. We suppose that, for all $j \in \N$, we have uniformly in $\lambda \geq 1$
\begin{equation*}
\big\vert \partial_{x}^{j} \psi_{\lambda} (x) \big\vert \lesssim 1,
\end{equation*}
and that there exist an open interval $I \Subset ] 0, + \infty [$ and constants $\delta , B , C >0$ such that
\begin{equation*}
\forall x \geq B \qquad \psi_{\lambda} (x) = C ,
\end{equation*}
and
\begin{equation*}
\forall x \in I \qquad \psi_{\lambda}^{\prime} (x) \geq \delta ,
\end{equation*}
for all $\lambda \geq 1$. The aim of this section is to establish a Mourre estimate for the operators $\psi_{\lambda} (\lambda P)$ which holds uniformly in $\lambda \gg 1$. We will assume the hypotheses of Theorem \ref{th1}. Let
\begin{equation*}
A= \frac{1}{2} ( x D + D x ) , \quad D ( A ) = \big\{ u \in L^{2} ( \R^{d} ) ; \ A u \in L^{2} ( \R^{d} ) \big\} ,
\end{equation*}
be the generator of dilations.

\begin{proposition}\sl  \label{PM1}
$i)$ For all $j \in \N$ and $\varepsilon > 0$, we have $\psi_{\lambda} (\lambda P) \in C^{j} (A)$ and
\begin{equation}
\big\Vert \ad^j_{A} \psi_{\lambda} (\lambda P) \big\Vert \lesssim
\left\{\begin{aligned}
&1 && \text{if } d \geq 3,   \\
&\lambda^{\varepsilon} && \text{if } d=2 .
\end{aligned} \right. 
\end{equation}

$ii)$ For $\lambda$ large enough, we have the following Mourre estimate
\begin{equation}  \label{M7}
\one_{I} ( \lambda P) i \big[ \psi_{\lambda}(\lambda P), A \big] \one_{I} ( \lambda P) \geq \delta ( \inf I ) \one_{I} ( \lambda P ) .
\end{equation}

$iii)$ For all $\varphi \in C^{\infty}_{0}( \R )$ and $\mu , \varepsilon > 0$, we have
\begin{equation}
\big\Vert \<A\>^{\mu} \varphi (\lambda P) \<x\>^{-\mu} \big\Vert \lesssim \lambda^{- \frac{1}{2} \min ( \mu , d/2 ) + \varepsilon} .
\end{equation}
\end{proposition}

The rest of this section is devoted to the proof of the previous proposition. We start by proving the commutator estimate.

\begin{lemma}\sl \label{lemma1}
For all $j \in \N$ and $\varepsilon > 0$, we have $\psi_{\lambda} (\lambda P) \in C^{j} (A)$ and
\begin{equation*}
\big\Vert \ad^j_{A} \psi_{\lambda} (\lambda P) \big\Vert \lesssim
\left\{\begin{aligned}
&1 && \text{if } d \geq 3,   \\
&\lambda^{\varepsilon} && \text{if } d=2.
\end{aligned} \right. 
\end{equation*}
\end{lemma}

\begin{proof}
It is well known that $P \in C^{1} (A)$. Moreover, using the pseudodifferential calculus, we obtain $\psi_{\lambda} (\lambda P) \in C^{j} (A)$ for all $j \in \N$. We now estimate the multi-commutators. In the following, a term $r_{j}$, $j\in \N$, will denote a smooth function such that
\begin{equation} \label{c17}
\forall \alpha \in \N^d \qquad \partial^{\alpha}_{x} r_{j} (x) = \CO \big( \< x \>^{-\rho - j - \vert \alpha \vert} \big) .
\end{equation}
Also let $\widetilde{\partial}_j=\partial_j b$ and $R$ be a term of the form
\begin{equation*}
R = \widetilde{\partial}^{*} r_{0} \widetilde{\partial} + \widetilde{\partial}^{*} r_{1} + r_{1} \widetilde{\partial}+r_2 ,
\end{equation*}
where, to clarify the statement, we have not written the sums over the indexes. Then, a direct calculation (see also \cite[(3.22) and page 42]{BoHa10_01}) gives
\begin{equation} \label{c7}
i [P , A ] = ( 2 P + R), \qquad [R,A] = R .
\end{equation}

Since $\psi_{\lambda} ( x ) = C$ for $x \geq B$, there exists $\chi_{\lambda} \in C^{\infty}_{0} ( [- 2 B , 2 B ] ; \R )$, uniformly bounded in $\lambda$ with all its derivatives, such that
\begin{equation*}
\psi_{\lambda} ( x ) = C + \chi_{\lambda} ( x ) ,
\end{equation*}
for $x \in [ 0 , + \infty [$. In particular, $\psi_{\lambda} ( \lambda P ) = C + \chi_{\lambda} ( \lambda P )$ because $\lambda P \geq 0$. Let $\widetilde{\chi}_{\lambda} \in C^{\infty}_{0} ( \C )$ be an almost analytic extension of $\chi_{\lambda}$ supported in a fixed compact of $\C$ with
\begin{equation*}
\forall n \in \N \qquad \big\vert \overline{\partial} \widetilde{\chi}_{\lambda} (z) \big\vert \lesssim \vert \im z \vert^{n} ,
\end{equation*}
uniformly in $\lambda \geq 1$. Therefore, using \eqref{c7} and the Helffer--Sj\"ostrand formula, $\ad_{A}^{j} \psi_{\lambda} (\lambda P)$ is a finite sum of terms of the form
\begin{equation*}
\int \overline{\partial} \widetilde{\chi}_{\lambda} (z) (\lambda P-z)^{-1} \prod_{1}^{k} \Big( \lambda (P+R) (\lambda P-z)^{-1} \Big) L (d z) ,
\end{equation*}
for some $k \leq j$. Since $\lambda P ( \lambda P -z)^{-1} = 1 + z ( \lambda P -z)^{-1}$, $\ad_{A}^{j} \psi_{\lambda} (\lambda P)$ can be written as a finite sum of terms of the form
\begin{equation} \label{a22}
\int z^{\ell} \overline{\partial} \widetilde{\chi}_{\lambda} (z) (\lambda P-z)^{- n_{0}} \lambda R (\lambda P-z)^{- n_{1}} \cdots \lambda R (\lambda P-z)^{- n_{k}} L (d z) ,
\end{equation}
with $k , \ell \leq j$ and $n_{\bullet} \in \N \setminus \{ 0 \}$.

Using Proposition \ref{a12} and Remark \ref{a54}, we see that we have for some $C>0$ 
\begin{equation} \label{R_1}
\big\Vert r ( \lambda P-z)^{-1} \big\Vert + \big\Vert ( \lambda P-z)^{-1} r^{*} \big\Vert \lesssim \frac{1}{\sqrt{\lambda} \vert \im z \vert^C}
\left\{ \begin{aligned}
&\lambda^{-\varepsilon} && d\geq 3,  \\
&\lambda^{\varepsilon} && d=2 ,
\end{aligned} \right.
\end{equation}
for all $\varepsilon>0$ small enough. Here $r$ (resp. $r^{*}$) is one of the operators $\< x \>^{- \rho /2} \widetilde{\partial}$ or $\< x \>^{-1-\rho/2}$ (resp. $\widetilde{\partial}^* \< x \>^{- \rho /2}$ or $\< x \>^{- 1 - \rho /2}$). In the same manner we find, using also Lemma \ref{reshalf},
\begin{equation} \label{R_2} 
\big\Vert r (\lambda P-z)^{-1} r^{*} \big\Vert \lesssim \frac{\lambda^{\varepsilon}}{\lambda \vert \im z\vert^C}
\end{equation}
Putting together \eqref{R_1}, \eqref{R_2} and $R = r^{*} \CO (1) r$, we get 
\begin{equation*}
\big\Vert (\lambda P-z)^{- \alpha_{0}} \lambda R (\lambda P-z)^{- \alpha_{1}} \cdots \lambda R (\lambda P-z)^{- \alpha_{k}} \big\Vert\lesssim \frac{1}{\vert \im z \vert^C}
\left\{ \begin{aligned}
&\lambda^{-\varepsilon} && d\geq 3,  \\
&\lambda^{\varepsilon} && d=2 .
\end{aligned} \right.
\end{equation*}
Combining with \eqref{a22}, this finishes the proof of the lemma.
\end{proof}

We now can prove the Mourre estimate
\begin{lemma}\sl
For $\lambda$ large enough, we have the following Mourre estimate
\begin{equation*}
\one_{I} ( \lambda P) i \big[ \psi_{\lambda}(\lambda P), A \big] \one_{I} ( \lambda P) \geq \delta ( \inf I ) \one_{I} ( \lambda P ) .
\end{equation*}
\end{lemma}

\begin{proof}
We take the same notations as in the proof of Lemma \ref{lemma1}. Then, we have
\begin{align*}
i \big[ \psi_{\lambda} (\lambda P) , A \big] &= - \frac{1}{\pi} \int_{\C} \overline{\partial} \widetilde{\chi}_{\lambda} (z) ( \lambda P -z)^{-1} i[ \lambda P , A] ( \lambda P - z)^{-1} L ( d z)    \\
&= - \frac{1}{\pi} \int_{\C} \overline{\partial} \widetilde{\chi}_{\lambda} (z) ( \lambda P -z)^{-1} \lambda ( 2 P + \widehat{R}) ( \lambda P -z)^{-1} L ( d z) ,
\end{align*}
with, see \cite[(3.22)]{BoHa10_01},
\begin{equation*}
\widehat{R} = \widetilde{\partial}^{*} r_{0} \widetilde{\partial} + \widetilde{\partial}^{*} r_{1} + r_{1} \widetilde{\partial} .
\end{equation*}
Therefore, we obtain
\begin{align*}
\one_{I} ( \lambda P) i \big[ \psi_{\lambda} (\lambda P), A \big] \one_{I} ( \lambda P) &= 2 \one_{I} ( \lambda P) \psi^{\prime}_{\lambda} ( \lambda P) \lambda P \one_{I} ( \lambda P) + \one_{I} ( \lambda P) \widetilde{R} \one_{I} (\lambda P) \\
&\geq 2 \delta ( \inf I ) \one_{I} ( \lambda P) + \one_{I} ( \lambda P) \widetilde{R} \one_{I} (\lambda P) ,
\end{align*}
with
\begin{equation*}
\widetilde{R} = - \frac{1}{\pi} \int_{\C} \overline{\partial} \widetilde{\chi}_{\lambda} (z) ( \lambda P -z)^{-1} \lambda \widehat{R} ( \lambda P -z)^{-1} L (d z) .
\end{equation*}
By Proposition \ref{a12}, we have for some $\varepsilon , C >0$
\begin{equation*}
\big\Vert (\lambda P-z)^{-1} \lambda \widehat{R} ( \lambda P - z)^{-1} \big\Vert \lesssim \frac{\lambda^{- \varepsilon}}{\vert \im z \vert^C} .
\end{equation*}
Then $\Vert \widetilde{R} \Vert \lesssim \lambda^{- \varepsilon}$ and we get the Mourre estimate if $\lambda$ is sufficiently large.
\end{proof}

\begin{lemma}\sl
For all $\varphi \in C^{\infty}_{0}( \R )$ and $\mu , \varepsilon > 0$, we have
\begin{equation}
\big\Vert \<A\>^{\mu} \varphi (\lambda P) \<x\>^{-\mu} \big\Vert \lesssim \lambda^{- \frac{1}{2} \min ( \mu , d/2 ) + \varepsilon} .
\end{equation}
\end{lemma}

\begin{proof}
Here, we use the previous notations. We write
\begin{equation} \label{jb}
\big\Vert \<A\>^{\mu} \varphi (\lambda P) u \big\Vert \lesssim \big\Vert \< A \>^{\mu - [ \mu ]} A^{ [ \mu ]} \varphi (\lambda P) u \big\Vert + \big\Vert \< A \>^{\mu -[ \mu ]} \varphi (\lambda P) u \big\Vert .
\end{equation}

$\bullet$ We start by estimating the first term in \eqref{jb}. In the following, $\widehat{\varphi}$ will always be a function of the form $\widehat{\varphi} (x) = c(x+1)^{n} \varphi (x)$ with $c\in \C$ and $n \in \N$. The values of $c$ and $n$ can change from line to line. We first prove by induction over $[ \mu ] \in \N$ that, for all $\varphi \in C^{\infty}_{0} ( \R )$,
\begin{align}
A^{[\mu]} \varphi (\lambda P) = \sum_{J + K \leq [\mu]} \sum_{\fract{n_{\bullet} \geq 1}{\text{finite}}} (\lambda P+1)^{- n_{0}} \prod_{j=1}^{J} & \Big( \lambda (P+R) (\lambda P+1)^{- n_{j}} \Big)  \nonumber  \\
(\lambda & P + 1)^{-1} \prod_{k=1}^{K} \Big( A (\lambda P+1)^{-1} \Big) \widehat{\varphi} (\lambda P) . \label{a25}
\end{align}
For $[ \mu ] = 0$, we have
\begin{equation*}
\varphi (\lambda P) = (\lambda P+1)^{-2} (\lambda P+1)^{2} \varphi (\lambda P) = (\lambda P+1)^{-1} (\lambda P+1)^{-1} \widehat{\varphi} (\lambda P) .
\end{equation*}
Assume now that \eqref{a25} holds until some $[ \mu ] \geq 0$. We write $A^{[\mu] +1} \varphi (\lambda P) = A A^{[\mu]} \varphi (\lambda P)$ and try to commute $A$ with the right hand side of \eqref{a25}. If $A$ commutes really, we get
\begin{align*}
\sum_{J + K \leq [\mu]} \sum_{\fract{n_{\bullet} \geq 1}{\text{finite}}} (\lambda P+1)^{- n_{0}} \prod_{j=1}^{J} & \Big( \lambda (P+R) (\lambda P+1)^{- n_{j}} \Big)   \\
&(\lambda P + 1)^{-1} \prod_{k=1}^{K} \Big( A (\lambda P+1)^{-1} \Big) A (\lambda P+1)^{-1} (\lambda P+1) \widehat{\varphi} (\lambda P) ,
\end{align*}
which is of the required type. Its remains to study the commutator. We remark that $[A , P +R ] = 2 i ( P +R )$ and
\begin{equation} \label{a23}
\big[ A , (\lambda P+1)^{-1} \big] = - 2 i (\lambda P+1)^{-1} \lambda ( P +R ) (\lambda P+1)^{-1} .
\end{equation}
In particular, the commutator between $A$ and $(\lambda P+1)^{- n_{0}} \prod_{j=1}^{J} \big( \lambda (P+R) (\lambda P+1)^{- n_{j}} \big)$ can be written as a finite sum of terms of the form
\begin{equation*}
(\lambda P+1)^{- \widetilde{n}_{0}} \prod_{j=1}^{\widetilde{J}} \Big( \lambda (P+R) (\lambda P+1)^{- \widetilde{n}_{j}} \Big) ,
\end{equation*}
with $\widetilde{J} = J$ or $\widetilde{J} = J+1$, which gives terms of the required type. Now, using \eqref{a23}, the commutator between $A$ and $(\lambda P + 1)^{-1} \prod_{k=1}^{K} \big( A (\lambda P+1)^{-1} \big)$ can be written as a finite sum of terms of the form
\begin{equation} \label{a33}
(\lambda P + 1)^{-1} \prod_{k=1}^{K_{1}} \Big( A (\lambda P+1)^{-1} \Big) \lambda ( P +R ) (\lambda P+1)^{-1} \prod_{k=1}^{K_{2}} \Big( A (\lambda P+1)^{-1} \Big) ,
\end{equation}
with $K_{1} + K_{2}= K$. Then, we commute to the right the $K$ operators $A$ of this equation. Using \eqref{a23} and $[A , P +R ] = 2 i ( P +R )$, \eqref{a33} becomes
\begin{equation*}
\sum_{\fract{\widehat{J} + \widehat{K} \leq K + 1}{\widehat{J} \geq 1}} \sum_{\fract{\widehat{n}_{\bullet} \geq 1}{\text{finite}}} (\lambda P+1)^{- \widehat{n}_{0}} \prod_{j=1}^{\widehat{J}} \Big( \lambda (P+R) (\lambda P+1)^{- \widehat{n}_{j}} \Big) A^{\widehat{K}} .
\end{equation*}
But, since $\widehat{K} \leq K \leq [ \mu ]$, we can apply the induction hypothesis to $A^{\widehat{K}} \widehat{\varphi} (\lambda P)$ and this term will contribute as the required type. Summing up, we have obtained \eqref{a25} for $[ \mu ] +1$ and then for all $[ \mu ] \in \N$. As $\lambda P ( \lambda P+1)^{-1} = 1 - ( \lambda P+1)^{-1}$, we get from \eqref{a25}
\begin{align}
A^{[\mu]} \varphi (\lambda P) = \sum_{J + K \leq [\mu]} \sum_{\fract{n_{\bullet} \geq 1}{\text{finite}}} (\lambda P+1)^{- n_{0}} \prod_{j=1}^{J} & \Big( \lambda R (\lambda P+1)^{- n_{j}} \Big)  \nonumber  \\
&(\lambda P + 1)^{-1} \prod_{k=1}^{K} \Big( A (\lambda P+1)^{-1} \Big) \widehat{\varphi} (\lambda P) . \label{a34}
\end{align}

From $A = b^{-1} x D b + x (D b^{-1}) b - i d /2 = \CO ( \< x \> ) \widetilde{\partial} + \CO (1)$ and Proposition \ref{a12}, we obtain
\begin{gather*}
\big\Vert A (\lambda P+1)^{- n} ( \lambda^{1/2} \widetilde{\partial}^{*} ) \< x \>^{-1} \big\Vert \lesssim \lambda^{-1/2 + \varepsilon} ,   \\
\big\Vert (\lambda P+1)^{- n} ( \lambda^{1/2} \widetilde{\partial}^{*} ) \< x \>^{-1} \big\Vert \lesssim \lambda^{-1/2+ \varepsilon} ,
\end{gather*}
for all $\varepsilon > 0$ and $n \in \N \setminus \{ 0 \}$ which gives first
\begin{equation*}
\big\Vert \< A \> (\lambda P + 1)^{- n} ( \lambda^{1/2} \widetilde{\partial}^{*} ) \< x \>^{-1} \big\Vert \lesssim \lambda^{-1/2 + \varepsilon} ,
\end{equation*}
and then, for all $0 \leq \nu \leq 1$,
\begin{equation} \label{a35}
\big\Vert \< A \>^{\nu} (\lambda P + 1)^{- n} ( \lambda^{1/2} \widetilde{\partial}^{*} ) u \big\Vert \lesssim \lambda^{- \nu /2 + \varepsilon} \big\Vert \< x \>^{\nu} u \big\Vert ,
\end{equation}
by interpolation. In a similar way, we estimate
\begin{equation*}
\big\Vert \< A \> (\lambda P+1)^{- n} \lambda^{1/2} \< x \>^{-1} \big\Vert \lesssim \lambda^{\varepsilon} .
\end{equation*}
By interpolation we obtain, for all $0 \leq \nu \leq 1$,
\begin{equation} \label{a36}
\big\Vert \< A \>^{\nu} (\lambda P+1)^{-1} \lambda^{1/2} \< x \>^{-1} u \big\Vert \lesssim \lambda^{\varepsilon} \Vert u \Vert .
\end{equation}
The same way, Proposition \ref{a13} yields
\begin{equation} \label{a43}
\big\Vert \< A \>^{\nu} \varphi ( \lambda P ) \< x \>^{- \nu} u \big\Vert \lesssim \lambda^{- \nu / 2 + \varepsilon} \Vert u \Vert .
\end{equation}Using Proposition \ref{a12} and Lemma \ref{reshalf}, we obtain by interpolation for $0 \leq  \nu \leq 1$
\begin{equation}
\begin{gathered}
\big\Vert \< x \>^{\nu} ( \lambda^{1/2} \widetilde{\partial} ) (\lambda P+1)^{-n} u \big\Vert \lesssim \lambda^{\varepsilon} \big\Vert \< x \>^{\nu} u \big\Vert ,     \\
\big\Vert \< x \>^{\nu} \lambda^{1/2} \< x \>^{-1} (\lambda P+1)^{- n} u \big\Vert \lesssim \lambda^{\nu /2 + \varepsilon} \Vert u \Vert ,  \\
\big\Vert \< x \>^{\nu} ( \lambda^{1/2} \widetilde{\partial} ) (\lambda P+1)^{-n} ( \lambda^{1/2} \widetilde{\partial}^{*} ) u \big\Vert \lesssim \lambda^{\varepsilon} \big\Vert \< x \>^{\nu} u \big\Vert ,     \\
\big\Vert \< x \>^{\nu} ( \lambda^{1/2} \widetilde{\partial} ) (\lambda P+1)^{-n} \lambda^{1/2} \< x \>^{-1} u \big\Vert \lesssim \lambda^{\nu /2 + \varepsilon} \Vert u \Vert ,  \\
\big\Vert \< x \>^{\nu} \lambda^{1/2} \< x \>^{-1} (\lambda P+1)^{- n} ( \lambda^{1/2} \widetilde{\partial}^{*} ) u \big\Vert \lesssim \lambda^{\varepsilon} \big\Vert \< x \>^{\nu} u \big\Vert ,       \\
\big\Vert \< x \>^{\nu} \lambda^{1/2} \< x \>^{-1} (\lambda P+1)^{- n} \lambda^{1/2} \< x \>^{-1} u \big\Vert \lesssim \lambda^{\nu /2 + \varepsilon} \Vert u \Vert .  \label{a37}
\end{gathered}
\end{equation}
We remark that $\lambda R$ is a finite sum of terms of the form $\widetilde{r}^{*} \CO (1 ) \widetilde{r}$ where $\widetilde{r}$ (resp. $\widetilde{r}^{*}$) is one of the operators $\lambda^{1/2} \< x \>^{-1}$ or $\lambda^{1/2} \widetilde{\partial}$ (resp. $\lambda^{1/2} \< x \>^{-1}$ or $\lambda^{1/2} \widetilde{\partial}^{*}$). Then, putting together \eqref{a35}, \eqref{a36} and \eqref{a37}, we find
\begin{equation} \label{a38}
\bigg\Vert \< A \>^{\mu -[ \mu ]} (\lambda P+1)^{- n_{0}} \prod_{j=1}^{J} \Big( \lambda R (\lambda P+1)^{- n_{j}} \Big) \bigg\Vert \lesssim \lambda^{-\frac{\mu -[ \mu ]}{2} + \varepsilon} \big\Vert \< x \>^{\mu -[ \mu ]} u \big\Vert + \lambda^{\varepsilon} \Vert u \Vert .
\end{equation}

It remains to estimate
\begin{equation*}
\< x \>^{\nu} (\lambda P + 1)^{-1} \prod_{k=1}^{K} \Big( A (\lambda P+1)^{-1} \Big) \widehat{\varphi} (\lambda P) ,
\end{equation*}
for $\nu = 0$ and $\nu = \mu -[ \mu ]$. As $A = \widetilde{\partial}^{*} \CO ( \< x \> ) + \CO (1)$, Proposition \ref{a12} gives, for $\alpha \geq d/2$,
\begin{align*}
\big\Vert \< x \>^{\alpha} (\lambda P+1)^{-1} A u \big\Vert \lesssim{}& \lambda^{-1/2 + \varepsilon} \big\Vert \< x \>^{\alpha +1} u \big\Vert + \lambda^{\varepsilon} \big\Vert \< x \>^{\alpha} u \big\Vert  \\
&+ \lambda^{\alpha /2-d/4-1/2 + \varepsilon} \big\Vert \< x \>^{d/2+1} u \big\Vert + \lambda^{\alpha /2-d/4 + \varepsilon} \big\Vert \< x \>^{d/2} u \big\Vert ,
\end{align*} 
and, for $\alpha \leq d/2$,
\begin{equation*}
\big\Vert \< x \>^{\alpha} (\lambda P+1)^{-1} A u \big\Vert \lesssim \lambda^{-1/2 + \varepsilon} \big\Vert \< x \>^{\alpha +1} u \big\Vert + \lambda^{\varepsilon} \big\Vert \< x \>^{\alpha} u \big\Vert .
\end{equation*} 
Therefore, we obtain
\begin{equation*}
\big\Vert \< x \>^{\alpha} (\lambda P+1)^{-1} A u \big\Vert \lesssim \sum_{\fract{2 a + b = \alpha}{0 \leq b \leq \alpha +1}} \lambda^{a + \varepsilon} \big\Vert \<x\>^{b} u \big\Vert .
\end{equation*}
Combining with \eqref{a34} and \eqref{a38}, this gives
\begin{align}
\big\Vert \< A \>^{\mu -[ \mu ]} A^{[ \mu ]} \varphi (\lambda P) u \big\Vert \lesssim{}& \sum_{\fract{2 a+b= \mu -[ \mu ]}{0\leq b \leq \mu}} \lambda^{- \frac{\mu -[ \mu ]}{2}} \lambda^{a + \varepsilon} \big\Vert \< x \>^{b} \widehat{\varphi} (\lambda P) u \big\Vert  \nonumber  \\
&+ \sum_{\fract{2a^{\prime} + b^{\prime} =0}{0 \leq b^{\prime} \leq [ \mu ]}} \lambda^{a^{\prime} + \varepsilon} \big\Vert \< x \>^{b^{\prime}} \widehat{\varphi} (\lambda P) u \big\Vert  \nonumber   \\
=&: S_{1} + S_{2} . \label{a39}
\end{align}

$\star$ $1^{\text{st}}$ case: $\mu \geq d/2$. Using Proposition \ref{a13}, we obtain for $b < d/2$
\begin{equation*}
\big\Vert \< x \>^{b} \widehat{\varphi} (\lambda P) u \big\Vert \lesssim \lambda^{b/2 - d/4 + \varepsilon} \big\Vert \< x \>^{d/2} u \big\Vert .
\end{equation*}
Therefore, we get
\begin{equation} \label{a40}
S_1 \lesssim \sum_{\fract{2 a+b= \mu -[ \mu ]}{d /2 \leq b \leq \mu}} \lambda^{- \frac{\mu -[ \mu ]}{2}} \lambda^{a + \varepsilon} \big\Vert \< x \>^{b} u \big\Vert .
\end{equation}
Note that we have $a = \frac{- b + \mu -[ \mu ]}{2} \leq -d/4 + \frac{\mu -[ \mu ]}{2}$. In the same manner, we find
\begin{equation} \label{a41}
S_2 \lesssim \sum_{\fract{2a^{\prime} + b^{\prime} =0}{d/2 \leq b^{\prime} \leq [ \mu ]}} \lambda^{a^{\prime} + \varepsilon} \big\Vert \< x \>^{b^{\prime}} u \big\Vert .
\end{equation}
Here we have $a^{\prime} = - \frac{b^{\prime}}{2} \leq - d / 4$.

$\star$ $2^{\text{nd}}$ case: $\mu< d/2$. By Proposition \ref{a13} we have, for $\alpha \leq \mu$,
\begin{equation*}
\big\Vert \< x \>^{\alpha} \widehat{\varphi} (\lambda P) u \big\Vert \lesssim \lambda^{\frac{\alpha - \mu}{2} + \varepsilon} \big\Vert \< x \>^{\mu} u \big\Vert .
\end{equation*}
Therefore, we get
\begin{equation} \label{a42}
S_{1} \lesssim \lambda^{-\frac{\mu}{2} + \varepsilon} \big\Vert \< x \>^{\mu} u \big\Vert \quad \text{ and } \quad S_{2} \lesssim \lambda^{-\frac{\mu}{2} + \varepsilon} \big\Vert \< x \>^{\mu} u \big\Vert .
\end{equation}

Putting \eqref{a39}, \eqref{a40}, \eqref{a41} and \eqref{a42} together, we eventually obtain
\begin{equation} \label{a44}
\big\Vert \< A \>^{\mu -[ \mu ]} A^{[ \mu ]} \varphi (\lambda P) \< x \>^{- \mu} \big\Vert \lesssim \lambda^{- \frac{1}{2} \min ( \mu , d/2 ) + \varepsilon} .
\end{equation}

$\bullet$ Let us now estimate the second term in \eqref{jb}. Let $\widetilde{\varphi}$ be a function with the same properties as $\varphi$ such that $\widetilde{\varphi} \varphi = \varphi$. From \eqref{a43}, we get
\begin{align*}
\big\Vert \< A \>^{\mu -[ \mu ]} \varphi (\lambda P) \< x \>^{- \mu} \big\Vert &\lesssim \big\Vert \< A \>^{\mu -[ \mu ]} \widetilde{\varphi} (\lambda P) \< x \>^{- \mu +[ \mu ]} \big\Vert \big\Vert \< x \>^{\mu -[ \mu ]} \varphi (\lambda P) \< x \>^{- \mu} \big\Vert    \\
&\lesssim \lambda^{-\frac{\mu -[ \mu ]}{2} + \varepsilon} \big\Vert \< x \>^{\mu -[ \mu ]} \varphi (\lambda P) \< x \>^{- \mu} \big\Vert .
\end{align*}
We have to distinguish two cases:

$\star$ $1^{\text{st}}$ case: $\mu \geq d/2$. Then, we have by Proposition \ref{a12}
\begin{equation*}
\big\Vert \< x \>^{\mu -[ \mu ]} \varphi (\lambda P) \< x \>^{- \mu} \big\Vert \lesssim \lambda^{- \frac{d}{4} + \frac{\mu -[ \mu ]}{2} + \varepsilon} .
\end{equation*}

$\star$ $2^{\text{nd}}$ case: $\mu< d/2$. Again by Proposition \ref{a12}, we find
\begin{equation*}
\big\Vert \< x \>^{\mu -[ \mu ]} \varphi (\lambda P) \< x \>^{- \mu} \big\Vert \lesssim \lambda^{- \frac{[ \mu ]}{2} + \varepsilon} .
\end{equation*}
Putting everything together, we find
\begin{equation} \label{a45}
\big\Vert \< A \>^{\mu -[ \mu ]} \varphi (\lambda P) \< x \>^{- \mu} \big\Vert \lesssim \lambda^{- \frac{1}{2} \min ( \mu ,d/2 ) + \varepsilon} ,
\end{equation}
and the lemma follows from \eqref{jb}, \eqref{a44} and \eqref{a45}.
\end{proof}

\Subsection{Proof of Theorems \ref{th1} and \ref{th4}} \label{a21}

We start by proving Theorem \ref{th4}. We can clearly suppose that $a_{0} = 0$ and $a_{1} > 0$. We will make a dyadic decomposition of the low frequencies. To do so, we will consider $\varphi \in C^{\infty}_{0} ( ] 0 , + \infty [ )$ such that
\begin{equation} \label{c13}
\forall x \in ]0,1] \qquad \sum_{1 \leq \lambda \text{ dyadic}} \varphi ( \lambda x) = 1.
\end{equation}
To $\varphi$ we associate $\widetilde{\varphi} \in C^{\infty}_{0} ( ] 0 , + \infty [ )$ such that $\widetilde{\varphi} \varphi = \varphi$.

Let now $I \Subset ] 0 , + \infty [$ be an open interval such that
\begin{equation*}
\supp \varphi + [ - h , h ] \subset I ,
\end{equation*}
for some $h > 0$ small enough. From the assumptions on $f$, there exists, for $\lambda$ large enough, an increasing function $\psi_{\lambda} \in C^{\infty} ( \R )$, uniformly bounded in $\lambda$ with all its derivates, such that
\begin{equation} \label{a50}
\psi_{\lambda} (x) =
\left\{ \begin{aligned}
& 0 &&\text{for } x \leq 0 , \\
& a_{1} x^{\alpha} + \lambda^{- \nu} x^{\alpha + \nu} g \Big( \frac{x}{\lambda} \Big) \quad &&\text{for } x \in I ,  \\
& C &&\text{for } x \geq B ,
\end{aligned} \right.
\end{equation}
where $B , C >0$ are positive constants. In particular, there exists $\delta > 0$ such that
\begin{equation*}
\forall x \in I \qquad \psi_{\lambda}^{\prime} (x) \geq \delta ,
\end{equation*}
for $\lambda$ sufficiently large. Then, $\psi_{\lambda}$ satisfies all the assumptions of Section \ref{SME}.

On the other hand, since $\psi_{\lambda} (x) = a_{1} x^{\alpha} + \CO ( \lambda^{- \nu} )$ for $x \in I$, there exist open intervals $\widetilde{J}$ and $J$ such that
\begin{equation*}
\psi_{\lambda} ( \supp \varphi ) \subset \widetilde{J} \Subset J \subset \psi_{\lambda} (I) ,
\end{equation*}
for $\lambda$ large enough. Let $\tau \in C^{\infty}_{0} (J)$ be such that $\tau =1$ on $\widetilde{J}$. As $\psi_{\lambda}$ is increasing, we get
\begin{equation} \label{a46}
\varphi (x) \prec \tau ( \psi_{\lambda} (x) ) \prec \one_{J} ( \psi_{\lambda} (x) ) \prec \one_{I} (x) ,
\end{equation}
where $f \prec g$ means that $g =1$ near the support of $f$.

Since $\psi_{\lambda}$ satisfies the assumptions of Section \ref{SME}, we can apply Proposition \ref{PM1} $ii)$ and find that, for $\lambda$ large enough, the Mourre estimate \eqref{M7} holds on the interval $I$. Therefore a Mourre estimate for the operator $\psi_{\lambda}(\lambda P)$ holds on $J$:
\begin{align*}
\one_{J} (\psi_{\lambda} (\lambda P)) i \big[ \psi_{\lambda} (\lambda & P) , A \big] \one_{J} (\psi_{\lambda} (\lambda P))  \\
&= \one_{J} (\psi_{\lambda} (\lambda P)) \one_{I}(\lambda P) i \big[ \psi_{\lambda} (\lambda P) , A \big] \one_{I} (\lambda P) \one_{J} (\psi_{\lambda} (\lambda P))  \\
&\geq \delta (\inf I) \one_{J} (\psi_{\lambda} (\lambda P) ).
\end{align*}
Then, from Proposition \ref{PM1} $i)$ and Theorem \ref{a16}, there exists $\lambda_{0} \geq 1$ such that for all $\mu , \varepsilon > 0$
\begin{equation} \label{a47}
\big\Vert \< A \>^{- \mu} e^{i t \psi_{\lambda} (\lambda P)} \tau ( \psi_{\lambda} (\lambda P) ) u \big\Vert \lesssim \lambda^{\varepsilon} \< t \>^{- \mu} \big\Vert \< A \>^{\mu} u \big\Vert ,
\end{equation}
uniformly for $\lambda \geq \lambda_{0}$ and $t \in \R$.

Let us now note that it is sufficient to prove Theorem \ref{th4} for $\supp \chi \subset [- \varepsilon , \varepsilon ]$ for all $\varepsilon >0$ small enough. Indeed for $0 \notin \supp \chi$ we can divide the support of $\chi$ into a finite number of intervals and use directly the result of Hunziker, Sigal and Soffer on each of these intervals since $P$ has no eigenvalues (see e.g. Donnelly \cite[Corollary 5.4]{Do97_01}). Furthermore, for $\varepsilon >0$ small enough, we have
\begin{equation} \label{a49}
\chi (x) = \sum_{\lambda_{0} \leq \lambda \text{ dyadic}} \varphi ( \lambda x) \chi (x) .
\end{equation}
Note also that \eqref{a50} and \eqref{a46} imply
\begin{equation} \label{a48}
e^{i t f(P)} \varphi ( \lambda P) = e^{i t \lambda^{-\alpha} \psi_{\lambda} (\lambda P)} \tau ( \psi_{\lambda} ( \lambda P) ) \varphi ( \lambda P) .
\end{equation}
We have to distinguish two cases:

$\star$ $1^{\text{st}}$ case: $\alpha \leq 1$. Using \eqref{a47}, \eqref{a48} and Proposition \ref{PM1} $iii)$, we get
\begin{align*}
\big\Vert \< x \>^{- \frac{d}{2\alpha}} e^{i t \lambda^{\alpha} f(P)} \varphi(\lambda P) u \big\Vert
&\lesssim \big\Vert \< x \>^{- \frac{d}{2\alpha}} \widetilde{\varphi} (\lambda P) \< A \>^{\frac{d}{2\alpha}} \< A \>^{- \frac{d}{2\alpha}} e^{i t \lambda^{\alpha} f(P)} \varphi (\lambda P) u \big\Vert   \\
&\lesssim \lambda^{- \frac{d}{4} + \frac{\widetilde{\varepsilon}}{3}} \big\Vert \< A \>^{- \frac{d}{2\alpha}} e^{i t \psi_{\lambda}(\lambda P)} \tau ( \psi_{\lambda} (\lambda P)) \varphi (\lambda P) u \big\Vert   \\
&\lesssim  \lambda^{- \frac{d}{4} + \frac{2 \widetilde{\varepsilon}}{3}} \< t \>^{- \frac{d}{2\alpha}} \big\Vert \< A \>^{\frac{d}{2\alpha}} \varphi(\lambda P) \< x \>^{- \frac{d}{2\alpha}} \big\Vert \big\Vert \< x \>^{\frac{d}{2\alpha}} u \big\Vert  \\
&\lesssim \lambda^{- \frac{d}{2} + \widetilde{\varepsilon}} \< t \>^{- \frac{d}{2\alpha}} \big\Vert \< x \>^{\frac{d}{2\alpha}} u \big\Vert ,
\end{align*}
for all $\widetilde{\varepsilon} >0$. Using \eqref{a49}, we then estimate
\begin{align*}
\big\Vert \< x \>^{- \frac{d}{2\alpha}} e^{i t f(P)} \chi (P) u \big\Vert &\lesssim \sum_{\lambda_{0} \leq \lambda \text{ dyadic}} \big\Vert \< x \>^{- \frac{d}{2\alpha}} e^{i( \lambda^{-\alpha} t) \lambda^{\alpha} f (P)} \varphi (\lambda P) \chi (P) u \big\Vert    \\
&\lesssim \sum_{\lambda_{0} \leq \lambda \text{ dyadic}} \frac{\lambda^{- \frac{d}{2} + \widetilde{\varepsilon}}}{1 + ( \lambda^{-\alpha} t)^{\frac{d}{2\alpha} - \varepsilon}} \big\Vert \< x \>^{\frac{d}{2\alpha}} \chi (P) u \big\Vert  \\
&\lesssim \sum_{\lambda_{0} \leq \lambda \text{ dyadic}} \frac{\lambda^{\widetilde{\varepsilon} - \alpha \varepsilon}}{\lambda^{\frac{d}{2} - \varepsilon \alpha} + t^{\frac{d}{2\alpha} - \varepsilon}} \big\Vert \< x \>^{\frac{d}{2\alpha}} \chi (P) u \big\Vert      \\
&\lesssim \sum_{\lambda_{0} \leq \lambda \text{ dyadic}} \frac{\lambda^{\widetilde{\varepsilon} - \alpha \varepsilon}}{\< t \>^{\frac{d}{2\alpha} - \varepsilon}} \big\Vert \< x \>^{\frac{d}{2\alpha}} \chi (P) u \big\Vert   \\
&\lesssim  \<t\>^{- \frac{d}{2 \alpha} + \varepsilon} \big\Vert \< x \>^{\frac{d}{2 \alpha}} \chi (P) u \big\Vert \lesssim \< t \>^{- \frac{d}{2 \alpha} + \varepsilon} \big\Vert \< x \>^{\frac{d}{2 \alpha}} u \big\Vert ,
\end{align*}
where we have chosen $\widetilde{\varepsilon}$ small enough with respect to $\varepsilon$.

$\star$ $2^{\text{nd}}$ case: $\alpha> 1$. We proceed in the same manner. We first get
\begin{equation} \label{ag1}
\big\Vert \< x \>^{- \frac{d}{2}} e^{i t \lambda^{\alpha} f (P)} \varphi (\lambda P) u \big\Vert
\lesssim \lambda^{- \frac{d}{2} + \widetilde{\varepsilon}} \< t \>^{- \frac{d}{2}} \big\Vert \< x \>^{\frac{d}{2}} u \big\Vert .
\end{equation}
We then estimate
\begin{align*}
\big\Vert \< x \>^{- \frac{d}{2}} e^{i t f (P)} \chi (P) u \big\Vert &\lesssim \sum_{\lambda_{0} \leq \lambda \text{ dyadic}} \big\Vert \< x \>^{- \frac{d}{2}} e^{i ( \lambda^{- \alpha} t) \lambda^{\alpha} f (P)} \varphi ( \lambda P) \chi (P) u \big\Vert  \\
&\lesssim \sum_{\lambda_{0} \leq \lambda \text{ dyadic}} \frac{\lambda^{- \frac{d}{2} + \widetilde{\varepsilon}}}{1 + ( \lambda^{- \alpha} t)^{\frac{d}{2\alpha} - \varepsilon}} \big\Vert \< x \>^{\frac{d}{2}} \chi (P) u \big\Vert  \\
&\lesssim \sum_{\lambda_{0} \leq \lambda \text{ dyadic}} \frac{\lambda^{\widetilde{\varepsilon} - \alpha \varepsilon}}{\lambda^{\frac{d}{2} - \varepsilon \alpha} + t^{\frac{d}{2 \alpha} - \varepsilon}} \big\Vert \< x \>^{\frac{d}{2}} \chi (P) u \big\Vert    \\
&\lesssim \sum_{\lambda_{0} \leq \lambda \text{ dyadic}} \frac{\lambda^{\widetilde{\varepsilon} - \alpha \varepsilon}}{\< t \>^{\frac{d}{2 \alpha} - \varepsilon}} \big\Vert \< x \>^{\frac{d}{2}} \chi (P) u \big\Vert  \\
&\lesssim  \< t \>^{- \frac{d}{2 \alpha} + \varepsilon} \big\Vert \< x \>^{\frac{d}{2}} \chi (P) u \big\Vert \lesssim \< t \>^{- \frac{d}{2 \alpha} + \varepsilon} \big\Vert \< x \>^{\frac{d}{2}} u \big\Vert ,
\end{align*}
where we have again chosen $\widetilde{\varepsilon}$ small enough with respect to $\varepsilon$. This finishes the proof of Theorem \ref{th4}.

Let us now explain how Theorem \ref{th1} follows from Theorem \ref{th4}. The estimate \eqref{wave2} follows from the case $\alpha = 1/2$. To prove the estimate \eqref{wave1}, we have to notice that by dividing by $\sqrt{P}$ we loose an additional $\lambda^{1/2}$ on $\varphi ( \lambda P )$. To compensate this, we only use the  $\< \lambda^{-1/2} t \>^{1-d+ \varepsilon}$ decay instead of the $\< \lambda^{-1/2} t\>^{-d + \varepsilon}$ decay. To prove estimate \eqref{KG} we use the Taylor development of $\sqrt{1+P}$ close to zero and find $\alpha = 1$. The estimates \eqref{Schr} and \eqref{4Schr1} correspond to the case $\alpha = 1$, the estimate \eqref{4Schr2} corresponds to $\alpha = 2$.

\section{High frequency estimates} \label{a53}

\Subsection{Abstract setting} \label{a24}

Here we obtain a semiclassical result which is used in the next section to prove the high frequency estimates of Theorem \ref{a28}. For that, we use the semiclassical microlocal analysis (see Dimassi and Sj\"{o}strand \cite{DiSj99_01} for more details). We work with the $\sigma$-temperate metric
\begin{equation*}
\gamma = \frac{d x^{2}}{\< x \>^{2}} + \frac{d \xi^{2}}{\< \xi \>^{2}} .
\end{equation*}
For $m(x, \xi )$ a weight function, let $S_{h}(m)$ be the set of functions $f(x, \xi ;h) \in C^{\infty} (\R^{d} \times \R^{d} )$ such that
\begin{equation*}
\vert \partial_{x}^{\alpha} \partial_{\xi}^{\beta} f (x, \xi ; h) \vert \lesssim m (x, \xi ) \< x \>^{- \vert \alpha \vert} \< \xi \>^{- \vert \beta \vert} ,
\end{equation*}
for all $\alpha , \beta \in \N^{d}$. In fact, $S_{h} (m)$ is the space of semiclassical symbols of weight $m$ for the metric $\gamma$. For $f \in S_{h}( m)$, the pseudodifferential operator with symbol $f$ is given by
\begin{equation*}
\big( \oph (f) u \big) (x) = \frac{1}{( 2 \pi h)^{d}} \int e^{i ( x - y , \xi ) / h} f \Big( \frac{x+ y}{2} , \xi ; h \Big) u (y) \, d y \, d \xi .
\end{equation*}
Let $\Psi_{h} (m) = \oph ( S_{h}(m) )$ denote the set of semiclassical pseudodifferential operators whose symbols are in $S_{h} (m)$.

We consider $Q_{h} = \oph (q) \in \Psi_{h} (1)$ such that
\begin{equation*}
q (x , \xi ; h) = q_{0}(x , \xi ) + h^{\delta} \Psi_{h} (1) ,
\end{equation*}
for some $0 < \delta \leq 1$. Let $I \Subset ]0 , + \infty [$ be an open interval. We assume that for $(x, \xi ) \in q_{0}^{-1} (I)$, we have $q_{0} (x, \xi ) = ( p_{0} (x, \xi ) )^{\alpha}$ where $\alpha >0$ and $p_{0} = \sum_{j,k} b^{2} (x) G_{j,k} (x) \xi_{j} \xi_{k}$ is the principal symbol of $P$.

Since $P$ is non-trapping by assumption \eqref{a51}, the positive energies are non-trapping for $p_{0}$. Then, using a result of G\'erard and Martinez \cite{GeMa88_01} and a compactness argument, one can construct a symbol $a (x, \xi ) \in S_{h} ( \< x \> )$ such that $a (x , \xi ) = x \cdot \xi$ for $(x, \xi ) \in p_{0}^{-1} (I^{1/ \alpha} )$ with $\vert x \vert$ large enough and
\begin{equation} \label{a17}
\{ p_{0} , a \} > \varepsilon_{0} ,
\end{equation}
on $p_{0}^{-1} (I^{1/ \alpha} )$ with $\varepsilon_{0} >0$. We then define
\begin{equation*}
A_{h} = \oph ( a ) \in \Psi_{h} ( \< x \> ) .
\end{equation*}

\begin{proposition}\sl \label{a14}
$i)$ The operator $A_{h}$ is essentially self-adjoint on $C^{\infty}_{0} ( \R^{d} )$ and, for all $\mu \geq 0$,
\begin{equation*}
\big\Vert \< A_{h} \>^{\mu} u \big\Vert \lesssim \big\Vert \< x \>^{\mu} u \big\Vert .
\end{equation*}

$ii)$ We have $Q_{h} \in C^{\infty} ( A_{h} )$ and, for all $j \geq 1$,
\begin{equation*}
\big\Vert \ad^{j}_{A_{h}} Q_{h} \big\Vert \lesssim h .
\end{equation*}

$iii)$ There exists $\varepsilon >0$ such that, for all $J \Subset I$ and $h$ small enough,
\begin{equation} \label{a18}
\one_{J} (Q_{h}) i [ Q_{h} , A_{h} ] \one_{J} (Q_{h}) \geq \varepsilon h \one_{J} (Q_{h}) .
\end{equation}
\end{proposition}

\begin{proof}
$i)$ From Nelson's theorem (see Reed and Simon \cite[Theorem X.36]{ReSi78_01}) with the operator of comparison $\< x \>$, $A_{h}$ is essentially self-adjoint on $C^{\infty}_{0} ( \R^{d} )$. Moreover, for $\mu$ even, $\< A_{h} \>^{\mu} = (1 + A_{h}^{2} )^{\mu /2} \in \Psi_{h} ( \< x \>^{\mu} )$ by the pseudodifferential calculus. Then $\< A_{h} \>^{\mu} \< x \>^{- \mu} \in \Psi_{h} (1)$ is a bounded operator. The general case, $\mu \geq 0$, follows from an interpolation argument.

$ii)$ Since $Q_{h} \in \Psi_{h} (1)$, the Beals lemma shows that $( Q_{h} + i)^{-1} \in \Psi_{h} (1)$. In particular, $( Q_{h} + i)^{-1}$ preserves the domain of $\< x \>$. On the other hand, using $A_{h} \in \Psi_{h} ( \< x \> )$, the pseudodifferential calculus gives that $\ad^{j}_{A_{h}} Q_{h} \in \Psi_{h} ( h^{j} )$. In particular, by the Calderon and Vaillancourt theorem, we have
\begin{equation*}
\big\Vert \ad^{j}_{A_{h}} Q_{h} \big\Vert \lesssim h .
\end{equation*}

$iii)$ Let $\varphi \in C^{\infty}_{0} (I)$ with $\varphi = 1$ on $J$. By the functional calculus of the pseudodifferential operators, $\varphi (Q_{h})$ satisfies
\begin{equation*}
\varphi ( Q_{h} ) = \oph ( \varphi ( p_{0}^{\alpha} ) ) + h^{\delta} \Psi_{h} (1).
\end{equation*}
Thus, using the composition rules of pseudodifferential operators,
\begin{equation*}
\varphi ( Q_{h} ) i [Q_{h} , A_{h} ] \varphi ( Q_{h} ) = h \oph \big( \varphi ( p_{0}^{\alpha} ) \alpha p_{0}^{\alpha -1} \{ p_{0} , a \} \varphi ( p_{0}^{\alpha} ) \big) + h^{1 + \delta} \Psi_{h} (1) .
\end{equation*}
On the support of $\varphi ( p_{0}^{\alpha} )$, we have $\alpha p_{0}^{\alpha -1} \{ p_{0} , a \} \geq 2 \varepsilon$ from \eqref{a17}, for some $\varepsilon >0$. Then, the G{\aa}rding inequality implies
\begin{align*}
\varphi ( Q_{h} ) i [Q_{h} , A_{h} ] \varphi ( Q_{h} ) &\geq 2 \varepsilon h \oph \big( \varphi^{2} ( p_{0}^{\alpha} ) \big) - \CO ( h^{1 + \delta} )  \\
&= 2 \varepsilon h \varphi^{2} ( Q_{h} ) - \CO ( h^{1 + \delta} ) .
\end{align*}
Since $\varphi =1$ on $J$, we eventually obtain \eqref{a18} for $h$ small enough.
\end{proof}

\begin{proposition}\sl \label{a15}
Let $\varphi \in C^{\infty}_{0} ( I )$ and $\mu \geq 0$. Then, there exists $h_{0} > 0$ such that
\begin{equation*}
\big\Vert \< x \>^{- \mu} e^{i t Q_{h}} \varphi ( Q_{h} ) \< x \>^{- \mu} \big\Vert \lesssim \< h t \>^{- \mu} ,
\end{equation*}
uniformly for $0 < h < h_{0}$ and $t \in \R$.
\end{proposition}

\begin{proof}
We apply Theorem \ref{a16} of Hunziker, Sigal and Soffer. Reading carefully the paper \cite{HuSiSo99_01} and using Proposition \ref{a14}, one can see that, in this semiclassical setting,
\begin{equation} \label{a19}
\big\Vert \< A_{h} \>^{- \mu} e^{i t Q_{h}} \varphi ( Q_{h} ) \< A_{h} \>^{- \mu} \big\Vert \lesssim \< h t \>^{- \mu} ,
\end{equation}
for all $\mu >0$ and $h$ small enough. More precisely, \cite[Lemma 2.1]{HuSiSo99_01} holds with the right hand side of \cite[(2.1)]{HuSiSo99_01} multiplied by $h$. And in the proof of \cite[Theorem 1.1]{HuSiSo99_01}, $e^{i H t}$ is replaced by $e^{i H \frac{t}{h}}$. Now we can replace $\< A_{h} \>^{- \mu}$ by $\< x \>^{- \mu}$ in \eqref{a19} using Proposition \ref{a14} $i)$ and the proposition follows.
\end{proof}

\Subsection{Proof of Theorems \ref{glest} and \ref{a28}}

We start by proving Theorem \ref{a28}.

$i)$ Let $\widehat{\varphi} \in C^{\infty}_{0} ( ]0 , + \infty [)$ and $\widetilde{\varphi} \in C^{\infty} ( \R ; \R )$ be such that $\varphi (x) \prec \widehat{\varphi} (x^{\alpha}) \prec \widetilde{\varphi} ( x )$ ($f \prec g$ means that $g =1$ near the support of $f$). Then, we can write
\begin{equation} \label{a29}
e^{i t f (P)} \varphi ( h^{2} P ) = e^{i t h^{-2 \alpha} Q_{h}} \varphi ( h^{2} P ) \quad \text{ with } \quad Q_{h} = h^{2 \alpha} f (P) \widetilde{\varphi} ( h ^{2} P ) = k_{h} (h^{2} P) ,
\end{equation}
and $k_{h} (x) = ( x^{\alpha} + h^{2 \nu} x^{\alpha - \nu} g (h^{-2} x ) ) \widetilde{\varphi} ( x )$. We can take $\widetilde{\varphi}$ such that $k_{h}$ is an increasing function satisfying $k_{h} =0$ near $0$ and such that $k_{h} (x)$ is constant for $x$ large enough. In particular, using that $g \big( \frac{1}{x} \big) \in C^{\infty} ([ 0, 1 [)$, $k_{h}$ is $C^{\infty}$ and all its derivatives are bounded uniformly with respect to $h$. From the properties of $k_{h}$, the functional calculus of the pseudodifferential operators implies that $Q_{h} = k_{h} (h^{2} P)$ satisfies the assumptions of Section \ref{a24} with $I = \supp \widehat{\varphi}$. Then, Proposition \ref{a15} implies
\begin{equation} \label{a30}
\big\Vert \< x \>^{- \mu} e^{i t Q_{h}} \widehat{\varphi} ( Q_{h} ) \< x \>^{- \mu} \big\Vert \lesssim \< h t \>^{- \mu} .
\end{equation}
Moreover, the spectral theorem gives $\varphi ( h^{2} P ) = \widehat{\varphi} (Q_{h} ) \varphi ( h^{2} P )$ for $h$ small enough. Then, \eqref{a29} and \eqref{a30} yield
\begin{align}
\big\Vert \< x \>^{- \mu} e^{i t f (P)} \varphi ( h^{2} P ) \< x \>^{- \mu} \big\Vert &= \big\Vert \< x \>^{- \mu} e^{i t h^{- 2 \alpha} Q_{h}} \widehat{\varphi} ( Q_{h} ) \< x \>^{- \mu} \< x \>^{\mu} \varphi ( h^{2} P ) \< x \>^{- \mu} \big\Vert \nonumber \\
&\lesssim \big\< h^{1 -2 \alpha} t \big\>^{- \mu} ,
\end{align}
since $\< x \>^{\mu} \varphi ( h^{2} P ) \< x \>^{- \mu} = \CO (1)$ by semiclassical pseudodifferential calculus. This shows \eqref{a26}.

$ii)$ We now prove \eqref{a27} and assume first $\alpha \leq 1/2$. There exists $\varphi \in C^{\infty}_{0} ( ] 0 , + \infty [ ; [ 0 , + \infty [)$ such that
\begin{equation*}
\sum_{\fract{h^{2} \text{ dyadic}}{0 < h < 1}} \varphi^{3} (h^{2} x) = 1 ,
\end{equation*}
for $x \in [ 1 , + \infty [$. Let $h_{0} >0$ and $\chi \in C^{\infty}_{0}( \R )$ be such that \eqref{a26} holds for $0 < h < h_{0}$ and
\begin{equation} \label{a20}
( 1 - \chi ) (x) = ( 1 - \chi ) (x) \sum_{\fract{h^{2} \text{ dyadic}}{0 < h < h_{0}}} \varphi^{3} (h^{2} x) .
\end{equation}

From the functional calculus of the pseudodifferential operators, the support of the symbol of $\varphi ( h^{2} P ) \in \Psi_{h} (1)$ is inside $\supp \varphi ( p_{0} )$ modulo $\Psi_{h} ( h^{\infty} \< x , \xi \>^{- \infty} )$. Then, for $\varphi \prec \widetilde{\varphi} \in C^{\infty}_{0} ( ] 0 , + \infty [ )$,
\begin{equation*}
\< x \>^{\mu} \varphi ( h^{2} P ) = \varphi ( h^{2} P ) \< x \>^{\mu} + \CO (h) \widetilde{\varphi} ( h^{2} P ) \< x \>^{\mu} + \Psi_{h} \big( h^{\infty} \< x , \xi \>^{- \infty} \big) .
\end{equation*}
In particular,
\begin{align}
\big\Vert \< x \>^{- \mu} e^{i t f (P)} \varphi^{2} ( h^{2} P ) u \big\Vert \lesssim{}& \big\< h^{1 - 2 \alpha} t \big\>^{- \mu} \big\Vert \< x \>^{\mu} \varphi ( h^{2} P) u \big\Vert  \nonumber  \\
\lesssim{}&  h^{2 \alpha \mu - \mu} \< t \>^{- \mu} \big\Vert \< x \>^{\mu} \varphi ( h^{2} P) u \big\Vert  \nonumber  \\
\lesssim{}&  h^{2 \alpha \mu - \mu} \< t \>^{- \mu} \Big( \big\Vert \varphi ( h^{2} P) \< x \>^{\mu} u \big\Vert + h \big\Vert \widetilde{\varphi} (h^{2} P) \< x \>^{\mu} u \big\Vert  \nonumber  \\
&+ \big\Vert \Psi_{h} \big( h^{\infty} \< x , \xi \>^{- \infty} \big) u \big\Vert \Big)  \nonumber  \\
\lesssim{}& \< t \>^{- \mu} \Big( \big\Vert \widetilde{\varphi} ( h^{2} P) (h^{2} P)^{\alpha \mu - \mu /2} \varphi ( h^{2} P) \< P \>^{\mu /2 - \alpha \mu} \< x \>^{\mu} u \big\Vert  \nonumber  \\
&+ h \big\Vert \widetilde{\varphi} (h^{2} P) (h^{2} P)^{\alpha \mu - \mu /2} \< P \>^{\mu /2 - \alpha \mu} \< x \>^{\mu} u \big\Vert + h \big\Vert \< x \>^{\mu} u \big\Vert_{H^{\mu - 2 \alpha \mu} ( \R^{d} )} \Big)  \nonumber  \\
\lesssim{}& \< t \>^{- \mu} \big\Vert \varphi (h^{2} P) \< P \>^{\mu /2 - \alpha \mu} \< x \>^{\mu} u \big\Vert + h \< t \>^{- \mu} \big\Vert \< x \>^{\mu} u \big\Vert_{H^{\mu - 2 \alpha \mu} ( \R^{d} )}  \label{a31} \\
\lesssim{}& \< t \>^{- \mu} \big\Vert \< x \>^{\mu} u \big\Vert_{H^{\mu - 2 \alpha \mu} ( \R^{d} )}. \label{a32}
\end{align}
Using $\< x \>^{- \mu} \varphi ( h^{2} P ) = \varphi ( h^{2} P ) \< x \>^{- \mu} + \CO (h)  \< x \>^{- \mu}$, \eqref{a20},  \eqref{a31} and \eqref{a32}, we obtain
\begin{align*}
\big\Vert \< x & \>^{- \mu} e^{i t f(P)} ( 1 - \chi ) (P) u \big\Vert^{2} \\
&= \bigg\Vert \<x\>^{- \mu} ( 1 - \chi ) (P) \<x\>^{\mu} \sum_{\fract{h^{2} \text{ dyadic}}{0 < h < h_{0}}} \<x\>^{- \mu} e^{i t f(P)} \varphi^{3} (h^{2} P) u \bigg\Vert^{2}  \\
&\lesssim \bigg\Vert \sum_{\fract{h^{2} \text{ dyadic}}{0 < h < h_{0}}} \<x\>^{- \mu} e^{i t f(P)} \varphi^{3} (h^{2} P) u \bigg\Vert^{2}  \\
&\lesssim \bigg\Vert \sum_{\fract{h^{2} \text{ dyadic}}{0 < h < h_{0}}} \varphi (h^{2} P) \<x\>^{- \mu} e^{i t f(P)} \varphi^{2} (h^{2} P) u \bigg\Vert^{2} + \bigg( \sum_{\fract{h^{2} \text{ dyadic}}{0 < h < h_{0}}} h \bigg\Vert \<x\>^{- \mu} e^{i t f(P)} \varphi^{2} (h^{2} P) u \bigg\Vert \bigg)^{2}  \\
&\lesssim \sum_{\fract{h^{2} \text{ dyadic}}{0 < h < h_{0}}} \big\Vert \<x\>^{- \mu} e^{i t f(P)} \varphi^{2} (h^{2} P) u \big\Vert^{2} + \bigg( \sum_{\fract{h^{2} \text{ dyadic}}{0 < h < h_{0}}} h \< t \>^{- \mu} \big\Vert \< x \>^{\mu} u \big\Vert_{H^{\mu - 2 \alpha \mu} ( \R^{d} )} \bigg)^{2}  \\
&\lesssim \sum_{\fract{h^{2} \text{ dyadic}}{0 < h < h_{0}}} \< t \>^{- 2 \mu} \big\Vert \varphi (h^{2} P) \< P \>^{\mu /2 - \alpha \mu} \< x \>^{\mu} u \big\Vert^{2} + \< t \>^{- 2 \mu} \big\Vert \< x \>^{\mu} u \big\Vert_{H^{\mu - 2 \alpha \mu} ( \R^{d} )}^{2}  \\
&\lesssim \< t \>^{- 2 \mu} \bigg\Vert \sum_{\fract{h^{2} \text{ dyadic}}{0 < h < h_{0}}} \varphi (h^{2} P) \< P \>^{\mu /2 - \alpha \mu} \< x \>^{\mu} u \bigg\Vert^{2} + \< t \>^{- 2 \mu} \big\Vert \< x \>^{\mu} u \big\Vert_{H^{\mu - 2 \alpha \mu} ( \R^{d} )}^{2}  \\
&\lesssim \< t \>^{- 2 \mu} \big\Vert \< x \>^{\mu} u \big\Vert_{H^{\mu - 2 \alpha \mu} ( \R^{d} )}^{2} .
\end{align*}
Here, we have used the quasi-orthogonality of the $\varphi ( h^{2} P )$'s and
\begin{equation*}
\sum \big\Vert \varphi ( h^{2} P ) v \big\Vert^{2} \leq \Big\Vert \sum \varphi ( h^{2} P ) v \Big\Vert^{2} ,
\end{equation*}
since $\varphi \geq 0$.

In the case $\alpha > 1/2$, we use
\begin{equation*}
\big\< h^{1 - 2 \alpha} t \big\>^{- \mu} \lesssim h^{2 \alpha \mu - \mu} \vert t \vert^{- \mu} ,
\end{equation*}
and
\begin{align*}
h^{2 \alpha \mu - \mu} \big\Vert \varphi (h^{2} P )u \big\Vert &\lesssim \big\Vert h^{2 \alpha \mu - \mu} (P^{\alpha \mu - \mu/2} + i ) \< P \>^{\mu /2 - \alpha \mu} \varphi (h^{2} P ) u \big\Vert  \\
&\lesssim \big\Vert \< P \>^{\mu /2 - \alpha \mu} \varphi (h^{2} P ) u \big\Vert \lesssim \big\Vert \varphi (h^{2} P ) u \big\Vert_{H^{\mu - 2 \alpha \mu} ( \R^{d} )}^{2} .
\end{align*}
Since the rest of the proof is similar, we omit the details.

We now explain how Theorem \ref{glest} follows from Theorems \ref{th1} and \ref{a28}. We only prove \eqref{wave3} since the other estimates can be treated the same way. Let $\chi , \widetilde{\chi} , \widehat{\chi} \in C^{\infty}_{0} ( \R )$ be such that $\one_{\{ 0 \}} \prec \widehat{\chi} \prec \chi \prec \widetilde{\chi}$. Using \eqref{wave1} for the low frequencies, \eqref{a27} with $\mu = d-1$ and $\alpha = 1/2$ for the high frequencies and the pseudodifferential calculus, we obtain
\begin{align}
\Big\Vert \< x \>^{1-d} \frac{\sin t\sqrt{P}}{\sqrt{P}} u \Big\Vert_{H^1 ( \R^{d} )} \lesssim{}& \Big\Vert \< P \>^{1/2} \< x \>^{1-d} \frac{\sin t\sqrt{P}}{\sqrt{P}} u \Big\Vert   \nonumber  \\
\lesssim{}& \Big\Vert \< P \>^{1/2} \< x \>^{1-d} \frac{\sin t\sqrt{P}}{\sqrt{P}} \chi (P) u \Big\Vert  \nonumber \\
&+ \Big\Vert \< P \>^{1/2} \< x \>^{1-d} \frac{\sin t\sqrt{P}}{\sqrt{P}} (1- \chi ) (P) u \Big\Vert  \nonumber \\
\lesssim{}& \big\Vert \< P \>^{1/2} \< x \>^{1-d} \widetilde{\chi} (P) \< x \>^{d-1} \big\Vert \Big\Vert \< x \>^{1-d} \frac{\sin t\sqrt{P}}{\sqrt{P}} \chi (P) u \Big\Vert   \nonumber  \\
&+ \Big\Vert \< P \>^{1/2} \< x \>^{1-d} \frac{(1 - \widehat{\chi} ) (P)}{\sqrt{P}} \< x \>^{d-1} \Big\Vert \big\Vert \< x \>^{1-d} \sin t\sqrt{P} (1- \chi ) (P) u \big\Vert \nonumber   \\
\lesssim{}& \< t \>^{1 - d + \varepsilon} \big\Vert \< x \>^{d-1} u \big\Vert .
\end{align}

\section{Hardy type estimates} \label{a1}

In this section we prove some Hardy type estimates which are slight generalizations of those obtained in \cite{BoHa10_01}. These estimates hold in all dimensions $d \geq 1$. Note that Vasy and Wunsch \cite{VaWu09_01} have also obtained Hardy type estimates for scattering manifolds. We begin with a generalization of Lemma B.1 of \cite{BoHa10_01} to the case $\gamma + \beta /2 > d/4$.

\begin{lemma}\sl \label{a2}
Let $0 \leq \beta$, $0 \leq \gamma \leq \min ( 1 , d/4)$ and $0 \leq \delta \leq d /4$. Then, for all $\varepsilon >0$, we have
\begin{equation} \label{ha3}
\big\Vert \< x \>^{\beta} ( \lambda P_{0} + 1)^{-1} u \big\Vert \lesssim \lambda^{- \gamma + \varepsilon} \big\Vert \< x \>^{\beta + 2 \gamma} u \big\Vert + \lambda^{\beta /2 - \delta + \varepsilon} \big\Vert \< x \>^{2 \delta} u \big\Vert ,
\end{equation}
uniformly for $\lambda \geq 1$.
\end{lemma}

\begin{proof}
To obtain this result, we mimic the proof of Lemma B.1 of \cite{BoHa10_01}. From that Lemma, \eqref{ha3} holds without its last term if $\gamma + \beta /2 \leq d/4$. So, we can assume that $\gamma + \beta /2 > d/4$ and then $\beta > d/2 - 2$.

Using an explicit formula for the kernel of $( \lambda P_{0} +1 )^{-1}$ and some properties of the Hankel functions, it has been shown in \cite[(B.5)]{BoHa10_01} that the kernel of $\< x \>^{\beta} ( \lambda P_{0} +1 )^{-1}$, written $\< x \>^{\beta} k_{d} (x - y , \lambda )$, satisfies
\begin{equation} \label{ha5}
\big\vert \< x \>^{\beta} k_{d} (x-y , \lambda ) \big\vert \lesssim \ell_{1} (x-y , \lambda ) \< y \>^{\beta} + \ell_{2} (x-y , \lambda ) ,
\end{equation}
where, using the notation $r = \lambda^{- 1/2}  \vert x - y \vert$,
\begin{equation*}
\ell_{1} (x-y , \lambda ) = \lambda^{-\frac{d}{2}} e^{- r/2} g_{d} (r)
\quad \text{ and } \quad
\ell_{2} (x-y , \lambda ) = \lambda^{-\frac{d}{2} + \frac{\beta}{2}} e^{- r/2} r^{\beta} g_{d} (r) ,
\end{equation*}
with
\begin{equation*}
g_{d} (r) = \left\{ \begin{aligned}
& 1 && \text{for } d =1 , \\
& \< \ln r \> && \text{for } d =2 , \\
& r^{2 - d} && \text{for } d \geq 3 .
\end{aligned} \right.
\end{equation*}
It is then enough to estimate the operators $L_{1} , L_{2}$ whose kernels are $\ell_{1} ( x - y , \lambda ) , \ell_{2} ( x - y , \lambda )$. From \cite[(B.9)]{BoHa10_01}, we have
\begin{equation} \label{ha4}
\Vert L_{1} \< x \>^{\beta} u \Vert \lesssim \lambda^{- \gamma + \varepsilon} \big\Vert \< x \>^{\beta + 2 \gamma} u \big\Vert.
\end{equation}

Let us now consider $L_{2}$. If $\delta =0$, the Young inequality and \eqref{ha5} give
\begin{equation*}
\Vert L_{2} u \Vert_{2} = \Vert \ell_{2} *  u \Vert_{2} \leq \Vert \ell_{2} \Vert_{1} \Vert u \Vert_{2} \lesssim \lambda^{\frac{\beta}{2}} \Vert u \Vert_{2} ,
\end{equation*}
where $\Vert \cdot \Vert_{p}$ designs the standard norm on $L^{p} ( \R^{n} )$. Assume now that $\delta > 0$. One more time, the Young inequality implies
\begin{equation*}
\Vert L_{2} u \Vert_{2} = \Vert \ell_{2} *  u \Vert_{2} \leq \Vert \ell_{2} \Vert_{q} \Vert u \Vert_{p} ,
\end{equation*}
for $1 \leq p \leq 2$ and $\frac{1}{q} = \frac{3}{2} - \frac{1}{p}$. We choose $p = \frac{2 d}{d + 4 ( \delta - \varepsilon)}$ and we have $q= \frac{d}{d - 2 ( \delta - \varepsilon)}$. For $\varepsilon$ small enough and $0 < \delta \leq d / 4$, the condition $1 \leq p \leq 2$ is fulfilled. Moreover, the condition $\beta > d/2 - 2$ and \eqref{ha5} imply that $\ell_{2} \in L^{r}( \R^{d} )$ for all $1 \leq r \leq 2$. In particular, $\ell_{2} \in L^{q}( \R^{d} )$ and
\begin{equation} \label{a7}
\Vert L_{2} u \Vert_{2} \lesssim \lambda^{\frac{\beta}{2}- \frac{d}{2}} \lambda^{\frac{d}{2 q}} \Vert u \Vert_{p} = \lambda^{\frac{\beta}{2}- \delta + \varepsilon} \Vert u \Vert_{p} .
\end{equation}
Using the H\"{o}lder inequality, we obtain
\begin{equation} \label{a6}
\Vert u \Vert_{p} \leq \Big( \int \vert u \vert^{p s} \< x \>^{\alpha s} d x \Big)^{1/p s} \Big( \int \< x \>^{- \alpha t} d x \Big)^{1/p t} ,
\end{equation}
with $\frac{1}{s} + \frac{1}{t} =1$. We choose $s = \frac{2}{p} \geq 1$ and $\alpha = 2 p \delta$. In particular, $\alpha t = \frac{d \delta}{\delta - \varepsilon} > d$ and the last term in the previous estimate is finite. Then, \eqref{a6} becomes
\begin{equation} \label{a8}
\Vert u \Vert_{p} \lesssim \big\Vert \< x \>^{2 \delta} u \big\Vert_{2} .
\end{equation}
Combining \eqref{a7} and \eqref{a8}, we finally obtain
\begin{equation} \label{a9}
\Vert L_{2} u \Vert_{2} \lesssim \lambda^{\frac{\beta}{2}- \delta + \varepsilon} \big\Vert \< x \>^{2 \delta} u \big\Vert_{2}.
\end{equation}
and the lemma follows from \eqref{ha4} and \eqref{a9}.
\end{proof}

Now, using the same ideas and mimicking the proofs of Lemma B.2 and Lemma B.3 of \cite{BoHa10_01}, one can show the following estimates for the free Laplacian.

\begin{lemma}\sl \label{a10}
Let $j \in \{1 , \ldots , d \}$, $0 \leq \beta$, $0 \leq \gamma \leq \min ( 1/2 , d/4)$ and $0 \leq \delta \leq d /4$. Then, for all $\varepsilon >0$, we have
\begin{equation*}
\big\Vert \< x \>^{\beta} (\lambda^{1/2} \partial_{j}) ( \lambda P_{0} +1)^{- 1} u \big\Vert \lesssim \lambda^{- \gamma + \varepsilon} \big\Vert \< x \>^{\beta + 2 \gamma} u \big\Vert + \lambda^{\beta /2 - \delta + \varepsilon} \big\Vert \< x \>^{2 \delta} u \big\Vert ,
\end{equation*}
uniformly for $\lambda \geq 1$.
\end{lemma}

\begin{lemma}\sl \label{a11}
Let $j,k \in \{1 , \ldots , d \}$, $0 \leq \beta$ and $0 \leq \delta \leq d /4$. Then, for all $\varepsilon > 0$, we have
\begin{equation*}
\big\Vert \< x \>^{\beta} (\lambda^{1/2} \partial_{j}) ( \lambda P_{0} +1)^{- 1} (\lambda^{1/2} \partial_{k}) u \big\Vert \lesssim \lambda^{\varepsilon} \big\Vert \< x \>^{\beta} u \big\Vert + \lambda^{\beta /2 - \delta + \varepsilon} \big\Vert \< x \>^{2 \delta} u \big\Vert ,
\end{equation*}
uniformly for $\lambda \geq 1$.
\end{lemma}

Then, using resolvent equations as in Section B.2 and Section B.3 of \cite{BoHa10_01}, one can obtain the following results. Since the proofs are similar to the ones of that paper, we omit the details here.

\begin{proposition}\sl \label{a12}
Let $0 \leq \beta$, $0 \leq \gamma \leq \min ( 1 , d/4)$ and $0 \leq \delta \leq d /4$. Then, for all $\varepsilon >0$, we have
\begin{equation*}
\big\Vert \< x \>^{\beta} ( \lambda P +1)^{- 1} u \big\Vert \lesssim \lambda^{- \gamma + \varepsilon} \big\Vert \< x \>^{\beta + 2 \gamma} u \big\Vert + \lambda^{\beta /2 - \delta + \varepsilon} \big\Vert \< x \>^{2 \delta} u \big\Vert ,
\end{equation*}
uniformly for $\lambda \geq 1$.

Let $j \in \{1 , \ldots , d \}$, $0 \leq \beta$, $0 \leq \gamma \leq \min ( 1/2 , d/4)$ and $0 \leq \delta \leq d /4$. Then, for all $\varepsilon > 0$, we have
\begin{align*}
\big\Vert \< x \>^{\beta} ( \lambda P +1)^{- 1} (\lambda^{1/2} \widetilde{\partial}_{j}^{*} ) u \big\Vert & \lesssim \lambda^{- \gamma + \varepsilon} \big\Vert \< x \>^{\beta + 2 \gamma} u \big\Vert + \lambda^{\beta /2 - \delta + \varepsilon} \big\Vert \< x \>^{2 \delta} u \big\Vert ,  \\
\big\Vert \< x \>^{\beta} (\lambda^{1/2} \widetilde{\partial}_{j} ) ( \lambda P +1)^{- 1} u \big\Vert & \lesssim \lambda^{- \gamma + \varepsilon} \big\Vert \< x \>^{\beta + 2 \gamma} u \big\Vert + \lambda^{\beta /2 - \delta + \varepsilon} \big\Vert \< x \>^{2 \delta} u \big\Vert ,
\end{align*}
uniformly for $\lambda \geq 1$.

Let $j,k \in \{1 , \ldots , d \}$, $0 \leq \beta$ and $0 \leq \delta \leq d /4$. Then, for all $\varepsilon > 0$, we have
\begin{equation*}
\big\Vert \< x \>^{\beta} (\lambda^{1/2} \widetilde{\partial}_{j} ) ( \lambda P +1)^{- 1} (\lambda^{1/2} \widetilde{\partial}_{k}^{*} ) u \big\Vert \lesssim \lambda^{\varepsilon} \big\Vert \< x \>^{\beta} u \big\Vert + \lambda^{\beta /2 - \delta + \varepsilon} \big\Vert \< x \>^{2 \delta} u \big\Vert ,
\end{equation*}
uniformly for $\lambda \geq 1$.
\end{proposition}

\begin{remark}\sl \label{a54}
In the previous proposition, we can replace $( \lambda P + 1)^{-1}$ by $( \lambda P - z)^{-1}$ for $z$ in a compact set of $\C$ and $\im z \neq 0$. In that case, a loss of the form $\vert \im z \vert^{- C}$, $C>0$ appears in the estimates (see Remark B.9 of \cite{BoHa10_01}).
\end{remark}

\begin{proposition}\sl \label{a13}
Let $\chi \in C^{\infty}_{0} ( \R )$, $j,k \in \{1 , \ldots , d \}$, $0 \leq \beta$ and $0 \leq \gamma , \delta \leq d/4$. Then, for all $\varepsilon > 0$, we have
\begin{gather*}
\big\Vert \< x \>^{\beta} \chi ( \lambda P ) u \big\Vert \lesssim \lambda^{- \gamma + \varepsilon} \big\Vert \< x \>^{\beta + 2 \gamma} u \big\Vert + \lambda^{\beta /2 - \delta + \varepsilon} \big\Vert \< x \>^{2 \delta} u \big\Vert ,  \\
\big\Vert \< x \>^{\beta} (\lambda^{1/2} \widetilde{\partial}_{j}) \chi ( \lambda P ) u \big\Vert \lesssim \lambda^{- \gamma + \varepsilon} \big\Vert \< x \>^{\beta + 2 \gamma} u \big\Vert + \lambda^{\beta /2 - \delta + \varepsilon} \big\Vert \< x \>^{2 \delta} u \big\Vert ,  \\
\big\Vert \< x \>^{\beta} \chi ( \lambda P ) (\lambda^{1/2} \widetilde{\partial}_{j}^{*} ) u \big\Vert \lesssim \lambda^{- \gamma + \varepsilon} \big\Vert \< x \>^{\beta + 2 \gamma} u \big\Vert + \lambda^{\beta /2 - \delta + \varepsilon} \big\Vert \< x \>^{2 \delta} u \big\Vert ,  \\
\big\Vert \< x \>^{\beta} (\lambda^{1/2} \widetilde{\partial}_{j}) \chi ( \lambda P ) ( \lambda^{1/2} \widetilde{\partial}_{k}^{*} ) u \big\Vert \lesssim \lambda^{- \gamma + \varepsilon} \big\Vert \< x \>^{\beta + 2 \gamma} u \big\Vert + \lambda^{\beta /2 - \delta + \varepsilon} \big\Vert \< x \>^{2 \delta} u \big\Vert ,
\end{gather*}
uniformly for $\lambda \geq 1$.
\end{proposition}

\begin{remark}\sl \label{remprop}
If in addition we have $\gamma + \beta /2 \leq d /4$, then we can take $\delta = \beta /2 + \gamma$ in the above propositions and the second term in the right hand side of the estimates disappears. For example, we have
\begin{equation*}
\big\Vert \< x \>^{\beta} (\lambda P+1)^{-1} u \big\Vert \lesssim \lambda^{- \gamma + \varepsilon} \big\Vert \< x \>^{\beta +2 \gamma} u \big\Vert .
\end{equation*}
\end{remark}

We will also need the following

\begin{lemma}\sl \label{reshalf}
Assume $d \geq 2$ and let $0 \leq \mu \leq 1$. Then, for all $\varepsilon > 0$, we have
\begin{equation*}
\big\Vert ( \lambda P+1)^{-1/2} \<x\>^{- \mu} \big\Vert \lesssim \lambda^{- \mu/2 + \varepsilon} ,
\end{equation*}
uniformly for $\lambda \geq 1$.
\end{lemma}

\begin{proof}
We make a dyadic decomposition of the energies between $\lambda^{-1}$ and $1$.  There exist $f \in C^{\infty}_{0} ( \R )$, $g \in C^{\infty}_{0} ( ]0, + \infty [)$ and $h \in C^{\infty}(]0, + \infty [)$ such that
\begin{equation*}
f ( \lambda x) + \sum_{\fract{\mu \text{ dyadic}}{1 < \mu < \lambda}} g^{2} ( \mu x) + h(x) =1
\end{equation*}
for all $x \in [0, + \infty [$. Then, Proposition \ref{a13} and the spectral theorem give
\begin{align}
\big\Vert ( \lambda P +1 )^{-1/2} \< x \>^{- \mu} \big\Vert\leq{}& \big\Vert ( \lambda P +1 )^{-1/2} f ( \lambda P) \< x \>^{- \mu} \big\Vert + \sum_{\fract{\mu \text{ dyadic}}{1 < \mu < \lambda}} \big\Vert ( \lambda P +1 )^{-1/2} g^{2} ( \mu P) \< x \>^{- \mu} \big\Vert \nonumber \\
&+ \big\Vert ( \lambda P +1 )^{-1/2} h (P) \< x \>^{- \mu} \big\Vert   \nonumber \\
\leq{}& \big\Vert f ( \lambda P) \< x \>^{- \mu} \big\Vert + \sum_{\fract{\mu \text{ dyadic}}{1 < \mu < \lambda}} \big\Vert ( \lambda P +1 )^{- \mu /2 + \varepsilon} g ( \mu P) \big\Vert \big\Vert g ( \mu P) \< x \>^{- \mu} \big\Vert \nonumber \\
&+ \big\Vert ( \lambda P +1 )^{- \mu /2 + \varepsilon} h (P) \big\Vert  \nonumber \\
\lesssim{}& \lambda^{- \mu /2 + \varepsilon} + \sum_{\fract{\mu \text{ dyadic}}{1 < \mu < \lambda}} ( \lambda \mu^{-1})^{- \mu /2 + \varepsilon} \mu^{- \mu /2 + \varepsilon /2} + \lambda^{- \mu /2 + \varepsilon}  \nonumber \\
\lesssim{}& \lambda^{- \mu /2 + \varepsilon} \bigg( 1 + \sum_{\fract{\mu \text{ dyadic}}{1 < \mu < \lambda}} \mu^{- \varepsilon /2} \bigg) \lesssim \lambda^{- \mu /2 + \varepsilon},
\end{align}
which finishes the proof.
\end{proof}

{\sl Acknowledgments.} The authors were partially supported by ANR-08-BLAN-0228. The first author thanks the Bernoulli Center, EPFL, Lausanne, for a partial support during the program ``{\it Spectral and dynamical properties of quantum Hamiltonians}''.

\bibliographystyle{amsplain}
\providecommand{\bysame}{\leavevmode\hbox to3em{\hrulefill}\thinspace}
\providecommand{\MR}{\relax\ifhmode\unskip\space\fi MR }
\providecommand{\MRhref}[2]{%
  \href{http://www.ams.org/mathscinet-getitem?mr=#1}{#2}
}
\providecommand{\href}[2]{#2}

\end{document}